\documentclass{amsart}

\usepackage{amssymb,latexsym,amsmath}

\usepackage{hyperref}
\usepackage{xcolor}
\hypersetup{colorlinks=true, linkcolor = blue, citecolor=red}

\newtheorem{theorem}{Theorem}[section]
\newtheorem{lemma}[theorem]{Lemma}
\newtheorem{proposition}[theorem]{Proposition}
\newtheorem{corollary}[theorem]{Corollary}
\newtheorem{conjecture}[theorem]{Conjecture}

\theoremstyle{definition}
\newtheorem{definition}[theorem]{Definition}
\newtheorem{example}[theorem]{Example}
\newtheorem{notation}[theorem]{Notation}

\theoremstyle{remark}
\newtheorem{remark}[theorem]{Remark}

\numberwithin{equation}{section}

\newcommand{\abs}[1]{\left|#1\right|}
\newcommand{\norm}[1]{\left\|#1\right\|}

\def\Xint#1{\mathchoice
{\XXint\displaystyle\textstyle{#1}}%
{\XXint\textstyle\scriptstyle{#1}}%
{\XXint\scriptstyle\scriptscriptstyle{#1}}%
{\XXint\scriptscriptstyle\scriptscriptstyle{#1}}%
\!\int}
\def\XXint#1#2#3{{\setbox0=\hbox{$#1{#2#3}{\int}$}
\vcenter{\hbox{$#2#3$}}\kern-.5\wd0}}

\def\dashint{\Xint-}

\begin{document}

\title[Weighted Bergman projection]{Weighted Bergman projection on the Hartogs triangle}

\author{Liwei Chen}
\address{Department of Mathematics, Washington University in St. Louis, St. Louis, Missouri 63130}
\email{chenlw@math.wustl.edu}

\subjclass[2010]{32A07, 32A25, 32A50, 42B20, 42B25.}

\date{\today}

\keywords{Hartogs triangle, Bergman projection, $L^p$ regularity, $A_p^+$-condition}

\begin{abstract}
We prove the $L^p$ regularity of the weighted Bergman projection on the Hartogs triangle, where the weights are powers of the distance to the singularity at the boundary. The restricted range of $p$ is proved to be sharp. By using a two-weight inequality on the upper half plane with Muckenhoupt weights, we can consider a slightly wider class of weights.
\end{abstract}

\maketitle

\section{Introduction}

\subsection{Setup}
Let $\Omega$ be a domain in $\mathbb{C}^n$.

\begin{definition}
\label{defweight}
A measurable function $\mu$ is a \textit{weight} on $\Omega$, if $\mu>0$ almost everywhere and is locally integrable on $\Omega$.
\end{definition}

For $p\ge1$, we consider the weighted $L^p$ space
\[
L^p(\Omega,\mu)=\{f\,\,\mbox{measurable on }\Omega\,:\,\norm{f}_{L^p(\Omega,\mu)}<\infty\},
\]
where $\norm{\cdot}_{L^p(\Omega,\mu)}$ is the weighted $L^p$ norm defined by
\[
\norm{f}_{L^p(\Omega,\mu)}=\left(\int_{\Omega}\abs{f(z)}^p\mu(z)\,dV(z)\right)^{\frac{1}{p}}.
\]

Let $\mathcal{O}(\Omega)$ be the set of holomorphic functions on $\Omega$. For $p=2$, it is easy to see that, if $\mu$ is continuous and non-vanishing on $\Omega$, then the analytic subspace $A^2(\Omega,\mu)=L^2(\Omega,\mu)\cap\mathcal{O}(\Omega)$ is closed in $L^2(\Omega,\mu)$.

\begin{definition}
For a continuous and non-vanishing weight $\mu$ on $\Omega$, we define the \textit{weighted Bergman projection} $\mathcal{B}_{\Omega,\mu}$ on $\Omega$ with the weight $\mu$ to be the orthogonal projection from $L^2(\Omega,\mu)$ to $A^2(\Omega,\mu)$. The weighted Bergman projection is an integral operator
\[
\mathcal{B}_{\Omega,\mu}(f)(z)=\int_{\Omega}B_{\Omega,\mu}(z,\zeta)f(\zeta)\mu(\zeta)\,dV(\zeta),
\]
where $B_{\Omega,\mu}(z,\zeta)$ is the \textit{weighted Bergman kernel} with $(z,\zeta)\in\Omega\times\Omega$.
\end{definition}

\subsection{Results}
In this paper, we study the $L^p$ regularity of the weighted Bergman projection on the Hartogs triangle
\[
\mathbb{H}=\{(z_1,z_2)\in\mathbb{C}^2\,:\,\abs{z_1}<\abs{z_2}<1\}
\]
with the weight
\begin{equation}
\label{weight}
\mu(z)=\abs{z_2}^{s'}\abs{g(z_2)}^2,
\end{equation}
where $z\in\mathbb{H}$, $s'\in\mathbb{R}$ and $g$ is a non-vanishing holomorphic function on the unit disk $\mathbb{D}$. Note that on $\mathbb{H}$, $\abs{z_2}$ is comparable to $\abs{z}$.

We first consider the weight $\mu$ with $g\equiv1$ in \eqref{weight}.

\begin{theorem}
\label{LpH}
For $s'\in\mathbb{R}$ with the unique expression $s'=s+2k$, where $k\in\mathbb{Z}$ and $s\in(0,2]$, let $\mathcal{B}_{\mathbb{H},s'}$ be the weighted Bergman projection on $\mathbb{H}$ with the weight $\mu(z)=\abs{z_2}^{s'}$, where $z\in\mathbb{H}$.
\begin{enumerate}
\item{For $s'\in(-2,\infty)$, $\mathcal{B}_{\mathbb{H},s'}$ is $L^p$ bounded if and only if $p\in\big(\frac{s+2k+4}{s+k+2},\frac{s+2k+4}{k+2}\big)$.}
\item{For $s'\in[-5,-2]$, $\mathcal{B}_{\mathbb{H},s'}$ is $L^p$ bounded for $p\in(1,\infty)$.}
\item{For $s'\in(-6,-5)$, then $k=-3$ and $s\in(0,1)$, $\mathcal{B}_{\mathbb{H},s'}$ is $L^p$ bounded if and only if $p\in\big(2-s,\frac{2-s}{1-s}\big)$.}
\item{When $s'=-6$, $\mathcal{B}_{\mathbb{H},s'}$ is $L^p$ bounded for $p\in(1,\infty)$.}
\item{For $s'\in(-\infty,-6)$, $\mathcal{B}_{\mathbb{H},s'}$ is $L^p$ bounded if and only if $p\in\big(\frac{s+2k+4}{k+2},\frac{s+2k+4}{s+k+2}\big)$.}
\end{enumerate}
\end{theorem}

In particular, if we allow the weight $\mu$ on $\mathbb{H}$ to be unbounded, then we can shrink the range of $p$ arbitrarily.

\begin{corollary}
\label{LpHs<0}
Given any $p_0\in[1,2)$ with its conjugate exponent $p_0'$, let $\mu(z)=\abs{z_2}^{-(p_0+4)}$, where $z\in\mathbb{H}$. The weighted Bergman projection on $\mathbb{H}$ with the weight $\mu$ is $L^p(\mathbb{H},\mu)$ bounded if and only if $p\in(p_0,p_0')$.
\end{corollary}

\begin{remark}
A similar result holds for the $n$-dimensional generalized Hartogs triangle. See section \ref{sec3} for details.
\end{remark}

To consider a wider class of weights of the form in \eqref{weight}, inspired by the ideas in \cite{LS1,Z}, we use a different method and prove the following result.

\begin{theorem}
\label{LpHwider}
Assume that $p>1$. Let $\mu$ be of the form in \eqref{weight}. Suppose that the weighted Bergman projection $\mathcal{B}_{\mathbb{D},\abs{g}^2}$ on $\mathbb{D}$ with the weight $\abs{g}^2$ is $L^p\big(\mathbb{D},|g|^2\big)$ bounded if and only if $p\in(p_0,p_0')$ for some $p_0\ge1$, and suppose that the weighted Bergman projection $\mathcal{B}_{\mathbb{H},s'}$ on $\mathbb{H}$ with the weight $\lambda(z)=\abs{z_2}^{s'}$, where $z\in\mathbb{H}$, is $L^p(\mathbb{H},\lambda)$ bounded if and only if $p\in(p_1,p_1')$ for some $p_1\ge1$ as in Theorem \ref{LpH}. Then the weighted Bergman projection $\mathcal{B}_{\mathbb{H},\mu}$ on $\mathbb{H}$ with the weight $\mu$ is $L^p(\mathbb{H},\mu)$ bounded if $p\in(p_0,p_0')\cap(p_1,p_1')$.

In addition, if $(p_1,p_1')\subset(p_0,p_0')$ properly, then $\mathcal{B}_{\mathbb{H},\mu}$ is $L^p(\mathbb{H},\mu)$ bounded if and only if $p\in(p_1,p_1')$.
\end{theorem}

\begin{example}
As in \cite{Z}, if we take $g(z)=(z-1)^{\alpha}$ for some $\alpha>0$, then we see $(p_0,p_0')=\big(\frac{2\alpha+2}{\alpha+2},\frac{2\alpha+2}{\alpha}\big)$. By Theorem \ref{LpH}, when $s'\in[0,\infty)$, we have $(p_1,p_1')=\big(\frac{s+2k+4}{s+k+2},\frac{s+2k+4}{k+2}\big)$. So $\mathcal{B}_{\mathbb{H},\mu}$ is bounded if $p\in\big(\frac{2\alpha+2}{\alpha+2},\frac{2\alpha+2}{\alpha}\big)\bigcap\big(\frac{s+2k+4}{s+k+2},\frac{s+2k+4}{k+2}\big)$.
\end{example}

By using the new method, we can also study the $L^p$ regularity of the weighted Bergman projection on $\mathbb{H}$ mapping from one weighted space $L^p\big(\mathbb{H},\abs{z_2}^{s'}\big)$ to the other $L^p\big(\mathbb{H},\abs{z_2}^{t}\big)$, where $z\in\mathbb{H}$, $t\in\mathbb{R}$, and $s'\in\mathbb{R}$ with the unique expression $s'=s+2k$ for $k\in\mathbb{Z}$ and $s\in(0,2]$. For simplicity, we focus on the case $k\ge-1$.

\begin{theorem}
\label{LpHst}
For $s'\in\mathbb{R}$ with the unique expression $s'=s+2k$, where $k\in\mathbb{Z}$ and $s\in(0,2]$, let $\mathcal{B}_{\mathbb{H},s'}$ be the weighted Bergman projection on $\mathbb{H}$ with the weight $\abs{z_2}^{s'}$, where $z\in\mathbb{H}$. Assume that $p>1$, $k\ge-1$, and $t\in\mathbb{R}$. Then $\mathcal{B}_{\mathbb{H},s'}$ is $L^p$ bounded from $L^p\big(\mathbb{H},\abs{z_2}^{s'}\big)$ to $L^p\big(\mathbb{H},\abs{z_2}^{t}\big)$ if $p\in\big(\frac{s+2k+4}{s+k+2},\frac{t+4}{k+2}\big)$.

In addition, if $t-s'\le(2-s)p$, then $\mathcal{B}_{\mathbb{H},s'}$ is $L^p$ bounded from $L^p\big(\mathbb{H},\abs{z_2}^{s'}\big)$ to $L^p\big(\mathbb{H},\abs{z_2}^{t}\big)$ if and only if $p\in\big(\frac{s+2k+4}{s+k+2},\frac{t+4}{k+2}\big)$.

On the other hand, if $p\le\frac{s+2k+4}{s+k+2}$, then $\mathcal{B}_{\mathbb{H},s'}$ is unbounded for all $t\in\mathbb{R}$.
\end{theorem}

\subsection{Background}
The $L^p$ regularity of the (ordinary) Bergman projection is of considerable interest for many years. For domains with smooth boundary, we refer the readers to \cite{F,PS,NRSW,MS,CD} for the principal results. For the non-smooth case, see \cite{LS1,KP1,KP2,Z}.

In particular, several people recently are interested in the regularity of the Bergman projection on the Hartogs triangle. In \cite{CS1} and \cite{CS2}, Chakrabarti and Shaw focus on the $\overline{\partial}$-equation and the corresponding Sobolev regularity on the product domains and the Hartogs triangle. In \cite{C}, the author shows the (ordinary) Bergman projection is $L^p$ bounded on the Hartogs triangle if and only if $p\in\big(\frac{4}{3},4\big)$. In \cite{CZ}, Chakrabarti and Zeytuncu study the $L^p$ mapping property of the (ordinary) Bergman projection on the Hartogs triangle.

In contrast to the previous work, we consider the $L^p$ regularity of the weighted Bergman projection on the Hartogs triangle, where the weights are powers of the distance to the singularity at the boundary. The results may also provide an idea to determine the type of serious boundary singularity.

\subsection{Outline and Ideas}
By applying inflation principle\footnote{See section \ref{sec2} for details.}, working on the Hartogs triangle $\mathbb{H}$ is indeed equivalent to working on the punctured disk $\mathbb{D}^*=\mathbb{D}\setminus\{0\}$. Note that the weighted Bergman kernel ${B}_{s'}(z,\zeta)$ associated to $\mathbb{D}^*$ with the weight $\mu(z)=\abs{z}^{s'}$ can be expressed as a ``homotopy" between two weighted Bergman kernels
\[
B_{s'}(z,\zeta)=\frac{s}{2}B_{2k+2}(z,\zeta)+\left(1-\frac{s}{2}\right)B_{2k}(z,\zeta),
\]
where $(z,\zeta)\in\mathbb{D}^*\times\mathbb{D}^*$, $s'=s+2k$, $k\in\mathbb{Z}$, and $s\in(0,2]$.

In the first part of this article, we apply Schur's test to deduce the $L^p$ regularity of the weighted Bergman projection for $p$ being inside some open interval depending on the weight. In order to show the estimate is sharp, we then construct a sequence of $L^p$ functions whose image under the weighted Bergman projection blows up in the $L^p$ norm when $p$ equals to an endpoint of the interval.

We point out that in Theorem \ref{LpH} the range of $p$ does not change continuously as $s'$ varies. In fact, there are jumps around the even integers. The reason is that the analytic subspace $A^2\big(\mathbb{H},\abs{z_2}^{s'}\big)$ remains fixed as long as $s'$ does not go past the even integers.

In the second part, to consider a slightly wider class of weights, we apply the Cayley transform $\varphi:\mathbb{R}_+^2\to\mathbb{D}$ via $\varphi(z)=\frac{i-z}{i+z}$. Then we need to consider different types of the following two-weight inequality on the upper half plane
\begin{equation}
\label{2weight}
\int_{\mathbb{R}_+^2}\abs{\int_{\mathbb{R}_+^2}-\frac{f(w)}{(z-\overline{w})^2}\,dA(w)}^p\mu_1(z)\,dA(z)\le C\int_{\mathbb{R}_+^2}\abs{f(z)}^p\mu_2(z)\,dA(z),
\end{equation}
where $\mu_1$ and $\mu_2$ are two weights on $\mathbb{R}_+^2$.

With slight modifications of the proof of \cite[Proposition 4.5]{LS1}, we have the following proposition which is sufficient for our application.\footnote{See Definition \ref{Apclass} for the precise definition of $A_p^+(\mathbb{R}_+^2)$ weights.}

\begin{proposition}
\label{mu1>mu2}
For $p>1$, suppose that $\mu_1$ and $\mu_2$ are two weights on $\mathbb{R}_+^2$ such that $c\mu_1\ge\mu_2$ for some $c>0$. Then \eqref{2weight} holds for some $C>0$ if and only if $(\mu_1,\mu_2)\in A_p^+(\mathbb{R}_{+}^2)$.
\end{proposition}

Here we mention an open problem related to the two-weight inequality \eqref{2weight}. If one compares Theorem \ref{mu1>mu2} with \cite[Proposition 4.5]{LS1}, then one may suspect that the condition ``$c\mu_1\ge\mu_2$ for some $c>0$" could be redundant in the sufficient direction.

\begin{conjecture}
\label{conjecture1}
For $p>1$, if the two weights $\mu_1$ and $\mu_2$ satisfy $(\mu_1,\mu_2)\in A_p^+(\mathbb{R}_{+}^2)$, then \eqref{2weight} holds for some $C>0$.
\end{conjecture}

By considering the result in \cite{N}, we are also interested in the following variant.

\begin{conjecture}
For $p>1$, if the two weights $\mu_1$ and $\mu_2$ satisfy $(\mu_1^r,\mu_2^r)\in A_p^+(\mathbb{R}_{+}^2)$ for some $r>1$, then \eqref{2weight} holds for some $C>0$.
\end{conjecture}

\subsection*{Organization}
In section \ref{sec2}, we introduce the inflation principle in a general setting. In section \ref{sec3}, we prove a result for $\mathbb{D}^*$ (Proposition \ref{LpD*}) by applying Schur's test,  and then deduce Theorem \ref{LpH}. In section \ref{sec4}, we mainly focus on \eqref{2weight} and prove Proposition \ref{mu1>mu2}. In section \ref{sec5}, by applying Proposition \ref{mu1>mu2}, we prove another result for $\mathbb{D}^*$ (Proposition \ref{LpD*wider}) and hence deduce Theorem \ref{LpHwider}. As a  last application of Proposition \ref{mu1>mu2}, we prove Theorem \ref{LpHst} and study the mapping property of the weighted Bergman projection (Corollary \ref{sharpestimate}) in section \ref{sec6}.

\subsection*{Acknowledgements}
The content of this paper is a part of the author's Ph.D. thesis at Washington University in St. Louis (see \cite{Chen}). The author would like to thank his thesis advisor S.G. Krantz for very helpful comments and suggestions on his research. The author also wants to thank the referee for helpful recommendations to improve the presentation of the paper. Without all these input, this paper will not appear.

\section{The Inflation Principle}
\label{sec2}

\subsection{Preliminaries}
Let us temporarily consider the general setting for a moment, and suppose $\Omega$ is a domain in $\mathbb{C}^n$.
\begin{definition}
\label{inflation}
Let $m$ be an integer and let $z\in\mathbb{C}^m$. If $\mu$ is a non-vanishing weight on $\Omega$, we define the \textit{inflation} $\widetilde{\Omega}$ of $\Omega$ via $\mu$ by
\[
\widetilde{\Omega}=\{(z,w)\in\mathbb{C}^{m+n}\,:\,\abs{z}^2<\mu(w),w\in\Omega\}.
\]
Note that $\widetilde{\Omega}$ is a Hartogs domain.
\end{definition}

The verifications of the following two lemmas are straightforward.

\begin{lemma}
\label{induce_projection}
Let $F:X_1\to X_2$ be an isometry between two Banach spaces $X_1$ and $X_2$. Then it induces an isometry $F^*:\mathfrak{B}(X_1)\to\mathfrak{B}(X_2)$ between the spaces of the bounded operators by $F^*(T)=F\circ T\circ F^{-1}$ for any $T\in\mathfrak{B}(X_1)$.

In particular, suppose that $X_j=H_j$ is a Hilbert space, $j=1,2$. Let $S$ be a closed subspace of $H_1$, and let $P:H_1\to S$ be the orthogonal projection. Then $F$ induces an orthogonal decomposition $H_2=F(S)\oplus F(S^{\perp})$, that is, $F(S)$ is closed in $H_2$ and $F(S)^{\perp}=F(S^{\perp})$. Hence, $F^*(P):H_2\to F(S)$ is the orthogonal projection.
\end{lemma}

\begin{lemma}
\label{product_operator}
Suppose we have a weight $\mu_1>0$ on $\Omega_1$ and a weight $\mu_2>0$ on $\Omega_2$, both non-vanishing. Let $\mathcal{T}_1$ and $\mathcal{T}_2$ be the integral operators with kernels $T_1(w_1,\eta_1)$ on $\Omega_1\times\Omega_1$ and $T_2(w_2,\eta_2)$ on $\Omega_2\times\Omega_2$, respectively. That is,
\[
\begin{array}{c}
\mathcal{T}_1(f)(w_1)=\int_{\Omega_1}T_1(w_1,\eta_1)f(\eta_1)\mu_1(\eta_1)\,dV(\eta_1),\\
\mathcal{T}_2(g)(w_2)=\int_{\Omega_2}T_2(w_2,\eta_2)g(\eta_2)\mu_2(\eta_2)\,dV(\eta_2).
\end{array}
\]
Given any $p\in[1,\infty)$, if $\mathcal{T}_1$ is bounded on $L^p(\Omega_1,\mu_1)$ and $\mathcal{T}_2$ is bounded on $L^p(\Omega_2,\mu_2)$, then their product operator $\mathcal{T}=\mathcal{T}_1\otimes\mathcal{T}_2$ with kernel $T_1\otimes T_2$, is bounded on $L^p(\Omega_1\times\Omega_2,\mu_1\otimes\mu_2)$.

Conversely, assuming $\mathcal{T}_1$ and $\mathcal{T}_2$ both are non-trivial, if one of these two operator is unbounded, then $\mathcal{T}$ is unbounded.
\end{lemma}

\begin{corollary}
\label{transform_kernel}
Let $\Phi:\Omega_1\to\Omega_2$ be a biholomorphism between two domains in $\mathbb{C}^n$. Suppose $\Omega_j$ is equipped with the weight $\mu_j$, $j=1,2$, and $\mu_2=\mu_1\circ\Phi^{-1}$. Then we have the transformation formula for the weighted Bergman kernels
\[
B_{\Omega_1,\mu_1}(z,\zeta)=\det J_{\mathbb{C}}\Phi(z)B_{\Omega_2,\mu_2}(\Phi(z),\Phi(\zeta))\det\overline{J_{\mathbb{C}}\Phi(\zeta)},
\]
where $(z,\zeta)\in\Omega_1\times\Omega_1$.
\end{corollary}

\begin{proof}
Let $F:L^2(\Omega_1,\mu_1)\to L^2(\Omega_2,\mu_2)$ be the isometry by $F(f)=\det J_{\mathbb{C}}(\Phi^{-1})f\circ\Phi^{-1}$, for any $f\in L^2(\Omega_1,\mu_1)$. Then, by Lemma \ref{induce_projection}, we have $F^*(\mathcal{B}_{\Omega_1,\mu_1})=\mathcal{B}_{\Omega_2,\mu_2}$. By the uniqueness of the weighted Bergman kernel, we obtain the transformation formula above.
\end{proof}

\begin{corollary}
\label{Lp_induce_projection}
Let $\Phi:\Omega_1\to\Omega_2$ be a biholomorphism between two domains in $\mathbb{C}^n$. Suppose $\mathcal{B}_{\Omega_j,\mu_j}$ is the weighted Bergman projection on $\Omega_j$ with the weight $\mu_j$, $j=1,2$, and $\mu_2=\abs{\det J_{\mathbb{C}}(\Phi^{-1})}^2\mu_1\circ\Phi^{-1}$. Then, for $p\ge1$, $\mathcal{B}_{\Omega_1,\mu_1}$ is $L^p(\Omega_1,\mu_1)$ bounded if and only if $\mathcal{B}_{\Omega_2,\mu_2}$ is $L^p(\Omega_2,\mu_2)$ bounded.
\end{corollary}

\begin{proof}
This is a direct consequence of Lemma \ref{induce_projection}.
\end{proof}

\subsection{The Inflation Principle}
Now we are ready to prove the inflation principle, which generalizes the result \cite[Corollary 4.6]{Z}.

\begin{proposition}[Inflation Principle]
\label{inflation_theorem}
Let $\Omega\subset\mathbb{C}^n$ be a domain and let $\mu=\abs{g}^2$ for some non-vanishing holomorphic function $g$ on $\Omega$. Suppose that $\widetilde{\Omega}\subset\mathbb{C}^{m+n}$ is the inflation of $\Omega$ via $\mu$ as in Definition \ref{inflation}. Let $\lambda>0$ be a continuous function on $\Omega$. Then, for $p>1$, the weighted Bergman projection $\mathcal{B}_{\widetilde{\Omega},\lambda}$ on $\widetilde{\Omega}$ with the weight $\lambda$ is $L^p(\widetilde{\Omega},\lambda)$ bounded if and only if the weighted Bergman projection $\mathcal{B}_{\Omega,\mu^m\lambda}$ on $\Omega$ with the weight $\mu^m\lambda$ is $L^p(\Omega,\mu^m\lambda)$ bounded.
\end{proposition}

\begin{proof}
Since $g$ is holomorphic and non-vanishing, using the notation in Definition \ref{inflation}, we obtain the biholomorphism $\Phi:\widetilde{\Omega}\to\mathbb{B}^m\times\Omega$ via $\Phi(z,w)=\big(z/g(w),w\big)$, where $\mathbb{B}^m$ is the unit ball in $\mathbb{C}^m$.

A direct computation shows that $\abs{\det J_{\mathbb{C}}(\Phi^{-1})}^2=\mu^m$. By Corollary \ref{Lp_induce_projection}, Lemma \ref{product_operator}, and the fact that the Bergman projection $\mathcal{B}_{\mathbb{B}^m}$ on $\mathbb{B}^m$ is $L^p$-bounded for all $p\in(1,\infty)$, we see that $\mathcal{B}_{\widetilde{\Omega},\lambda}$ is $L^p$-bounded if and only if $\mathcal{B}_{\Omega,\mu^m\lambda}$ is $L^p$-bounded.
\end{proof}

\section{The Punctured Disk and the Hartogs Triangle}
\label{sec3}

\subsection{The Punctured Disk}
Using the notation in Proposition \ref{inflation_theorem}, if we take $\Omega=\mathbb{D}^*=\mathbb{D}\setminus\{0\}$ (the punctured disk) via the weight $\mu(w)=\abs{w}^2$, where $w\in\mathbb{D}^*$, then the inflation $\widetilde{\Omega}=\mathbb{H}$ is the Hartogs triangle. To prove Theorem \ref{LpH}, it suffices to consider the punctured disk $\mathbb{D}^*$ with the weight $\lambda(z)=\abs{z}^{s'}$, where $z\in\mathbb{D}^*$ and $s'\in\mathbb{R}$.

\begin{lemma}
For $s'\in\mathbb{R}$ with the unique expression $s'=s+2k$, where $k\in\mathbb{Z}$ and $s\in(0,2]$, the weighted Bergman kernel ${B}_{s'}(z,\zeta)$ associated to $\mathbb{D}^*$ with the weight $\lambda(z)=\abs{z}^{s'}$ has a ``homotopic" expression
\begin{equation}
\label{kernel}
\begin{split}
B_{s'}(z,\zeta)
&=\frac{s}{2}B_{2k+2}(z,\zeta)+\left(1-\frac{s}{2}\right)B_{2k}(z,\zeta)\\
&=\frac{s}{2}(z\overline{\zeta})^{-(k+1)}B_0(z,\zeta)+\left(1-\frac{s}{2}\right)(z\overline{\zeta})^{-k}B_0(z,\zeta),
\end{split}
\end{equation}
where $B_0(z,\zeta)$ is the (ordinary) Bergman kernel associated to the unit disk and $(z,\zeta)\in\mathbb{D}^*\times\mathbb{D}^*$. 
\end{lemma}

\begin{proof}
We first determine an orthonormal basis for the space $A^2\big(\mathbb{D}^*,\abs{z}^{s'}\big)$. Suppose $m,n\in\mathbb{Z}$; a direct computation shows,\footnote{We have normalized the area of $\mathbb{D}$ by setting $\mbox{Area}(\mathbb{D})=1$.} for $m+n+s'+2>0$,
\[
\int_{\mathbb{D}^*}z^n\overline{z}^m\abs{z}^{s'}\,dA(z)=\left\{ \begin{array}{rcl}
0, & \mbox{if} & n\neq m,\\
\frac{2}{2m+2+s'}, & \mbox{if} & n = m.
\end{array}\right.
\]
Therefore $\left\{\sqrt{\frac{2m+2+s'}{2}}z^m\right\}_{m>-(1+\frac{s'}{2})}$ is an orthonormal basis. So the corresponding weighted Bergman kernel is
\begin{equation}
\label{eq3.2}
\begin{split}
B_{s'}(z,\zeta)
&= \sum_{m>-(1+\frac{s'}{2})}\frac{2m+2+s'}{2}z^m\overline{\zeta}^m\\
&=\frac{(t+\frac{s'}{2})(z\overline{\zeta})^{t-1}-(t-1+\frac{s'}{2})(z\overline{\zeta})^t}{(1-z\overline{\zeta})^2},
\end{split}
\end{equation}
where $t$ is the smallest integer satisfying $t>-\frac{s'}{2}$.

From \eqref{eq3.2}, it is easy to verify the following equation
\begin{equation}
\label{eq3.3}
B_{s'+2}(z,\zeta)=(z\overline{\zeta})^{-1}B_{s'}(z,\zeta).
\end{equation}
Hence 2 is a ``period" of $s'$ for the weighted Bergman kernel $B_{s'}(z,\zeta)$. Let $s'=s\in(0,2]$. Then $t=0$, and from \eqref{eq3.2} we have
\begin{equation}
\label{eq3.4}
\begin{split}
B_s(z,\zeta)
&=\frac{\frac{s}{2}(z\overline{\zeta})^{-1}+(1-\frac{s}{2})}{(1-z\overline{\zeta})^2}\\
&=\frac{s}{2}(z\overline{\zeta})^{-1}B_0(z,\zeta)+\left(1-\frac{s}{2}\right)B_0(z,\zeta).
\end{split}
\end{equation}
Therefore, combining \eqref{eq3.3} and \eqref{eq3.4}, we obtain \eqref{kernel}.
\end{proof}

Following the idea in \cite{C}, we need three lemmas.

\begin{lemma}[Schur's Test]
\label{Schur}
Suppose $X$ is a measure space with a positive measure $\mu$. Let $T(x,y)$ be a positive measurable function on $X\times X$, and let $\mathcal{T}$ be the integral operator associated to the kernel function $T(x,y)$.

Given $p\in(1,\infty)$ with its conjugate exponent $p'$, if there exists a strictly positive function $h$ a.e. on $X$ and a constant $M>0$, such that
\begin{enumerate}
\item
$\int_{X}T(x,y)h(y)^{p'}\,d\mu(y)\le Mh(x)^{p'}$, for a.e. $x\in X$, and
\item
$\int_{X}T(x,y)h(x)^p\,d\mu(x)\le Mh(y)^p$, for a.e. $y\in X$. 
\end{enumerate}
Then $\mathcal{T}$ is bounded on $L^p(X,d\mu)$ with $\|\mathcal{T}\|\le M$.
\end{lemma}

\begin{proof}
See \cite[Theorem 4.1]{C}, or \cite{HKZ} for details.
\end{proof}

\begin{lemma}
\label{kerI}
For $-1<\alpha<0$ and $\beta>-2$, define
\[
I_{\alpha,\beta}(z)=\int_{\mathbb{D}^*}\frac{\big(1-\abs{\zeta}^2\big)^{\alpha}\abs{\zeta}^{\beta}\,dA(\zeta)}{\abs{1-z\overline{\zeta}}^2},
\]
where $z\in\mathbb{D}^*$. Then we have\footnote{The notation $A\approx B$ means there is a constant $c>0$ so that $c^{-1}B\le A\le cB$.} $I_{\alpha,\beta}(z)\approx\big(1-\abs{z}^2\big)^{\alpha}$, for any $z\in\mathbb{D}^*$.
\end{lemma}

\begin{proof}
See \cite[Lemma 3.3]{C}.
\end{proof}

\begin{lemma}
\label{unbounded}
Let $a_j=\left(\frac{1}{j}\right)^j$, $j=1,2,3,\dots$. For $p\ge1$, the sum $A_{n,p}=\sum_{j=1}^nj\left(a_j^{p/j}-a_{j+1}^{p/j}\right)$ diverges when $p=1$ and converges when $p>1$, as $n\to\infty$. More precisely, we have
\[
\lim_{n\to\infty}A_{n,1}=\infty
\]
and
\[
\lim_{n\to\infty}A_{n,p}\le c\sum_{j=1}^{\infty}\frac{1}{j^{1+\epsilon}}<\infty
\]
for all $p>1$ with some $c>0$ and sufficiently small $\epsilon>0$.
\end{lemma}

\begin{proof}
We first prove the following statement,\footnote{The notation $A\lesssim B$ means there is a constant $c>0$ so that $A\le cB$.}
\begin{equation}
\label{eq3.5}
\left(\frac{1}{j}\right)^2\lesssim\frac{1}{j}-\left(\frac{1}{j+1}\right)^{\frac{j+1}{j}}\lesssim\left(\frac{1}{j}\right)^{2-\epsilon'}
\end{equation}
for any $\epsilon'>0$, as $j\to\infty$.

We obtain the first inequality in \eqref{eq3.5} by looking at the limit (with L'H\^{o}pital's rule applied)
\begin{align*}
\lim_{j\to\infty}\frac{\frac{1}{j}-\left(\frac{1}{j+1}\right)^{\frac{j+1}{j}}}{\left(\frac{1}{j}\right)^2}
&=\lim_{j\to\infty}\frac{-\frac{1}{j^2}+\left(\frac{1}{j+1}\right)^{\frac{j+1}{j}}\left(-\frac{1}{j^2}\log(j+1)+\frac{1}{j}\right)}{-2\left(\frac{1}{j}\right)^3}\\
&=\frac{1}{2}\lim_{j\to\infty}\frac{1+\left(\frac{1}{j+1}\right)^{\frac{1}{j}}\left(\frac{\log(j+1)}{j+1}-\frac{j}{j+1}\right)}{\frac{1}{j}}\\
&=\frac{1}{2}\left[\lim_{j\to\infty}\frac{1-\left(\frac{1}{j+1}\right)^{\frac{1}{j}}\frac{j}{j+1}}{\frac{1}{j}}+\lim_{j\to\infty}\frac{\left(\frac{1}{j+1}\right)^{\frac{1}{j}}\frac{\log(j+1)}{j+1}}{\frac{1}{j}}\right]\\
&=\frac{1}{2}\left[\lim_{j\to\infty}\frac{j}{j+1}\cdot\lim_{j\to\infty}\frac{1+\frac{1}{j}-\left(\frac{1}{j+1}\right)^{\frac{1}{j}}}{\frac{1}{j}}+\lim_{j\to\infty}\left(\frac{1}{j+1}\right)^{\frac{1}{j}}\frac{j}{j+1}\log(j+1)\right]\\
&=\frac{1}{2}\left[\lim_{j\to\infty}\frac{1-\left(\frac{1}{j+1}\right)^{\frac{1}{j}}}{\frac{1}{j}}+1+\lim_{j\to\infty}\log(j+1)\right]\\
\mbox{(L'H\^{o}pital's rule)}
&=\frac{1}{2}\left[\lim_{j\to\infty}\frac{\left(\frac{1}{j+1}\right)^{\frac{1}{j}}\left(-\frac{1}{j^2}\log(j+1)+\frac{1}{j(j+1)}\right)}{-\frac{1}{j^2}}+1+\lim_{j\to\infty}\log(j+1)\right]\\
&=\lim_{j\to\infty}\log(j+1)\\
&=\infty.
\end{align*}
For the second inequality in \eqref{eq3.5}, a similar argument shows
\[
\lim_{j\to\infty}\frac{\frac{1}{j}-\left(\frac{1}{j+1}\right)^{\frac{j+1}{j}}}{\left(\frac{1}{j}\right)^{2-\epsilon'}}=0.
\]

Now, for $p=1$, we have
\[
A_{n,1}=\sum_{j=1}^nj\left[\frac{1}{j}-\left(\frac{1}{j+1}\right)^{\frac{j+1}{j}}\right].
\]
From \eqref{eq3.5}, we obtain
\[
\lim_{n\to\infty}A_{n,1}\gtrsim\sum_{j=1}^{\infty}j\cdot\left(\frac{1}{j}\right)^2=\sum_{j=1}^{\infty}\frac{1}{j}=\infty.
\]
For $p>1$, we consider the function $\psi(x)=x^p$, $x\in(0,1]$. By the mean-value theorem, for each $j$, we have
\[
\psi\left(\frac{1}{j}\right)-\psi\left(\left(\frac{1}{j+1}\right)^{\frac{j+1}{j}}\right)=\left[\frac{1}{j}-\left(\frac{1}{j+1}\right)^{\frac{j+1}{j}}\right]\psi'(x_j),
\]
where $\psi'(x)=px^{p-1}$ and $\left(\frac{1}{j+1}\right)^{\frac{j+1}{j}}\le x_j\le\frac{1}{j}$. Since $p-1>0$, we have
\[
x_j^{p-1}\le\left(\frac{1}{j}\right)^{p-1}.
\]
So we obtain
\[
\psi\left(\frac{1}{j}\right)-\psi\left(\left(\frac{1}{j+1}\right)^{\frac{j+1}{j}}\right)\le\left[\frac{1}{j}-\left(\frac{1}{j+1}\right)^{\frac{j+1}{j}}\right]p\left(\frac{1}{j}\right)^{p-1}.
\]
Therefore, from \eqref{eq3.5}, we have
\begin{align*}
\lim_{n\to\infty}A_{n,p}
&=\sum_{j=1}^{\infty}j\left(a_j^{\frac{p}{j}}-a_{j+1}^{\frac{p}{j}}\right)\\
&=\sum_{j=1}^{\infty}j\left[\psi\left(\frac{1}{j}\right)-\psi\left(\left(\frac{1}{j+1}\right)^{\frac{j+1}{j}}\right)\right]\\
&\le\sum_{j=1}^{\infty}j\left[\frac{1}{j}-\left(\frac{1}{j+1}\right)^{\frac{j+1}{j}}\right]p\left(\frac{1}{j}\right)^{p-1}\\
&\lesssim\sum_{j=1}^{\infty}j\left(\frac{1}{j}\right)^{2-\epsilon'}\left(\frac{1}{j}\right)^{p-1}\\
&=\sum_{j=1}^{\infty}\left(\frac{1}{j}\right)^{p-\epsilon'}\\
&<\infty,
\end{align*}
for sufficiently small $\epsilon'>0$, such that $p-\epsilon'=1+\epsilon$ for some $\epsilon>0$.
\end{proof}

With these lemmas in hand, we can study the $L^p$ regularity of the weighted Bergman projection on $\mathbb{D}^*$ with the weight $\lambda(z)=\abs{z}^{s'}$, where $z\in\mathbb{D}^*$ and $s'\in\mathbb{R}$.

\begin{proposition}
\label{LpD*}
For $s'\in\mathbb{R}$ with the unique expression $s'=s+2k$, where $k\in\mathbb{Z}$ and $s\in(0,2]$, let $\mathcal{B}_{s'}$ be the weighted Bergman projection on $\mathbb{D}^*$ with the weight $\lambda(z)=\abs{z}^{s'}$, where $z\in\mathbb{D}^*$.
\begin{enumerate}
\item{For $s'\in(0,\infty)$, $\mathcal{B}_{s'}$ is $L^p$ bounded if and only if $p\in\big(\frac{s+2k+2}{s+k+1},\frac{s+2k+2}{k+1}\big)$.}
\item{For $s'\in[-3,0]$, $\mathcal{B}_{s'}$ is $L^p$ bounded for $p\in(1,\infty)$.}
\item{For $s'\in(-4,-3)$, then $k=-2$ and $s\in(0,1)$, $\mathcal{B}_{s'}$ is $L^p$ bounded if and only if $p\in\big(2-s,\frac{2-s}{1-s}\big)$.}
\item{When $s'=-4$, $\mathcal{B}_{-4}$ is $L^p$ bounded for $p\in(1,\infty)$.}
\item{For $s'\in(-\infty,-4)$, $\mathcal{B}_{s'}$ is $L^p$ bounded if and only if $p\in\big(\frac{s+2k+2}{k+1},\frac{s+2k+2}{s+k+1}\big)$.}
\end{enumerate}
\end{proposition}

\begin{proof}
The proof here will be essentially the same as the proof of \cite[Theorem 1.2]{C}. So we will only give a brief outline here.\footnote{An alternative proof by using the two-weight inequality \eqref{2weight} can be found in section \ref{sec5}. See also \cite{Chen}.}

To prove the boundedness part, by \eqref{kernel}, we have
\[
\abs{B_{s'}(z,\zeta)}\le\abs{z\overline{\zeta}}^{-(k+1)}\abs{B_0(z,\zeta)},
\]
since $(z,\zeta)\in\mathbb{D}^*\times\mathbb{D}^*$. So it suffices to apply Lemma \ref{Schur} to the kernel
\[
T(z,\zeta)=\abs{z\overline{\zeta}}^{-(k+1)}\frac{1}{\abs{1-z\overline{\zeta}}^2}
\]
on $\mathbb{D}^*\times\mathbb{D}^*$ with the weight $\lambda$ and the positive function
\[
h(z)=(1-\abs{z}^2)^{\delta}\abs{z}^{\sigma}
\]
on $\mathbb{D}^*$ for some $\delta,\sigma\in\mathbb{R}$. By Lemma \ref{kerI}, we see that
\[
T(h^{p'})(z)\le Mh(z)^{p'}
\]
if $-1<\delta p'<0$, $-2<\sigma p'+s+k-1$, and $\sigma p'\le-(k+1)$. Similarly,
\[
T(h^p)(\zeta)\le Mh(\zeta)^p
\]
if $\delta\in\big(-\frac{1}{p},0\big)$ and $\sigma\in\big(-\frac{s+k+1}{p},-\frac{k+1}{p}\big]$. Therefore, such $\delta$ and $\sigma$ exist when
\begin{enumerate}
\item{$k\ge0$, $p\in\big(\frac{s+2k+2}{s+k+1},\frac{s+2k+2}{k+1}\big)$;}
\item{$k=-1$, $p\in(1,\infty)$;}
\item{$k=-2$ and $1\le s\le 2$, $p\in(1,\infty)$;}
\item{$k=-2$ and $0<s<1$, $p\in\big(2-s,\frac{2-s}{1-s}\big)$;}
\item{$k=-3$ and $s=2$, $p\in(1,\infty)$;}
\item{$k\le-3$ with (5) excluded, $p\in\big(\frac{s+2k+2}{k+1},\frac{s+2k+2}{s+k+1}\big)$.}
\end{enumerate}

To show the unboundedness part, we only need to look at $p=\frac{s+2k+2}{s+k+1}$. Let $a_j=\left(\frac{1}{j}\right)^j$, $j=1,2,3,\dots$, and define a function $g$ on $(0,1]$ such that $g(r)=r^{\frac{1}{j}-(s+k+1)}$, $r\in(a_{j+1},a_j]$. We consider the sequence
\[
f_n(z)=\left\{
\begin{array}{rc}
g(\abs{z})\left(\frac{\overline{z}}{\abs{z}}\right)^{k+1}, & \abs{z}\in(a_{n+1},1],\\
0, & \abs{z}\in[0,a_{n+1}].
\end{array}
\right.
\]
It is easy to see, $\{f_n\}\subset L^2(\mathbb{D}^*,\lambda)$. Since $p=\frac{s+2k+2}{s+k+1}>1$, by Lemma \ref{unbounded}, we have
\[
\norm{f_n}_{L^p(\mathbb{D}^*,\lambda)}^p=\frac{2}{p}A_{n,p}\le c\sum_{j=1}^{\infty}\frac{1}{j^{1+\epsilon}},
\]
for some $\epsilon>0$ and some $c>0$.

On the other hand, a direct computation shows
\[
\mathcal{B}_{s'}(f_n)(z)=sz^{-(k+1)}A_{n,1}.
\]
It is easy to see
\[
s+2k-(k+1)p=-2+\nu,
\]
for some $\nu>0$. So we obtain
\[
\norm{\mathcal{B}_{s'}(f_n)}_{L^p(\mathbb{D}^*,\lambda)}=s\left(\frac{2}{\nu}\right)^{\frac{1}{p}}A_{n,1}.
\]
Hence, by Lemma \ref{unbounded}, we see that
\[
\lim_{n\to\infty}\norm{\mathcal{B}_{s'}(f_n)}_{L^p(\mathbb{D}^*,\abs{z}^{s'})}=\infty.
\]
This completes the proof.
\end{proof}

\subsection{The Hartogs Triangle}
Now we can consider the Hartogs triangle and its generalization. We first prove Theorem \ref{LpH} and Corollary \ref{LpHs<0}.

\begin{proof}[Proof of Theorem \ref{LpH}.]
This is a direct consequence by combining Proposition \ref{inflation_theorem} and Proposition \ref{LpD*}.
\end{proof}

\begin{proof}[Proof of Corollary \ref{LpHs<0}.]
This is Theorem \ref{LpH} (2) (3) with $s'=-(p_0+4)$.
\end{proof}

If we apply Proposition \ref{inflation_theorem} several times, we will obtain the weighted version of the result in \cite{C}.

To be precise, for $j=1,\dots,l$, let $\Omega_j$ be a bounded smooth domain in $\mathbb{C}^{m_j}$ with a biholomorphic mapping $\phi_j:\Omega_j\to\mathbb{B}^{m_j}$ between $\Omega_j$ and the unit ball $\mathbb{B}^{m_j}$ in $\mathbb{C}^{m_j}$. We use the notation $\tilde{z_j}$ to denote the $j$th $m_j$-tuple in $z\in\mathbb{C}^{m_1+\cdots+m_l}$, that is $z=(\tilde{z_1},\dots,\tilde{z_l})$. Let $N=m_1+\cdots+m_l+n$, we define the $N$-dimensional Hartogs triangle by 
\[
\mathbb{H}_{\phi_j}^N=\Big\{(z,w)\in\mathbb{C}^{m_1+\cdots+m_l+n}\,:\,\max_{1\le j\le l}\abs{\phi_j(\tilde{z_j})}<\abs{w_1}<\abs{w_2}<\cdots<\abs{w_n}<1\Big\}.
\]
Let $\lambda(z,w)=\abs{w_1}^{s_1}\cdots\abs{w_n}^{s_n}$, where $(z,w)\in\mathbb{H}_{\phi_j}^N$ and $s_1,\dots,s_n\in\mathbb{R}$. We consider the weighted Bergman projection on $\mathbb{H}_{\phi_j}^N$ with the weight $\lambda$.

\begin{corollary}
\label{LpHN}
Let $\mathbb{H}_{\phi_j}^N$ and $\lambda$ be as above, and let $\mathcal{B}_{s'}$ be as in Proposition \ref{LpD*}. Then the weighted Bergman projection on $\mathbb{H}_{\phi_j}^N$ with the weight $\lambda$ is $L^p\big(\mathbb{H}_{\phi_j}^N,\lambda\big)$ bounded if and only if each of the following projections
\[
\mathcal{B}_{2(m_1+\cdots+m_l)+s_1},\mathcal{B}_{2(m_1+\cdots+m_l)+s_1+s_2+2},\cdots,\mathcal{B}_{2(m_1+\cdots+m_l)+s_1+\cdots+s_n+2(n-1)}
\]
is $L^p$ bounded on the corresponding weighted space.

In other words, assume that $p>1$ and for $j=1,2,\dots,n$ we let $I_j$ be one of the intervals for $p$ in Proposition \ref{LpD*}, so that the $j$th projection above is $L^p$ bounded if and only if $p\in I_j$. Then the weighted Bergman projection on $\mathbb{H}_{\phi_j}^N$ with the weight $\lambda$ is $L^p\big(\mathbb{H}_{\phi_j}^N,\lambda\big)$ bounded if and only if $p\in\cap I_j$.
\end{corollary}

To prove Corollary \ref{LpHN}, we need the following lemma.

\begin{lemma}
\label{LpH*}
Let $\mathbb{H}^{n*}=\{z\in\mathbb{C}^n\,:\,0<\abs{z_1}<\cdots<\abs{z_n}<1\}$ be the punctured $n$-dimensional Hartogs triangle. Suppose we have a weight $\lambda(z)=\abs{z_1}^{s_1}\cdots\abs{z_n}^{s_n}$ on $\mathbb{H}^{n*}$, where $s_1,\dots,s_n\in\mathbb{R}$. Then the weighted Bergman projection $\mathcal{B}_{\mathbb{H}^{n*},\lambda}$ is $L^p\big(\mathbb{H}^{n*},\lambda\big)$ bounded if and only if each of the following projections
\[
\mathcal{B}_{s_1},\,\,\,\mathcal{B}_{s_1+s_2+2},\,\,\,\cdots,\,\,\,\mathcal{B}_{s_1+\cdots+s_n+2(n-1)}
\]
is $L^p$ bounded on the corresponding space.
\end{lemma}

\begin{proof}
The conclusion will follow from the same argument as in the proof of Proposition \ref{inflation_theorem}, if we consider the biholomorphism $\Phi:\mathbb{H}^{n*}\to(\mathbb{D}^*)^{\times n}$ via $\Phi(z)=\big(\frac{z_1}{z_2},\cdots,\frac{z_{n-1}}{z_n},z_n\big)$ with $\abs{\det J_{\mathbb{C}}\Phi^{-1}(w)}^2=\abs{w_2w_3^2\cdots w_n^{n-1}}^2$, where $w\in(\mathbb{D}^*)^{\times n}$.
\end{proof}

Now we are ready to prove Corollary \ref{LpHN}.

\begin{proof}[Proof of Corollary \ref{LpHN}.]
If we iteratively apply Proposition \ref{inflation_theorem} $l$ times to $\Omega=\mathbb{H}^{n*}=\{w\in\mathbb{C}^n\,:\,0<\abs{w_1}<\cdots<\abs{w_n}<1\}$ with the same weight $\abs{w_1}^2$, then we will arrive at the space
\[
\mathbb{H}^N=\Big\{(z,w)\in\mathbb{C}^{m_1+\cdots+m_l+n}\,:\,\max_{1\le j\le l}\abs{\tilde{z_j}}<\abs{w_1}<\abs{w_2}<\cdots<\abs{w_n}<1\Big\}.
\]
So the weighted Bergman projection $\mathcal{B}_{\mathbb{H}^N,\lambda}$ is $L^p\big(\mathbb{H}^N,\lambda\big)$ bounded if and only if $\mathcal{B}_{\mathbb{H}^{n*}, \tilde{\lambda}}$ is $L^p\big(\mathbb{H}^{n*},\tilde{\lambda}\big)$ bounded, where $\lambda(z,w)=\abs{w_1}^{s_1}\cdots\abs{w_n}^{s_n}$ and $\tilde{\lambda}(z,w)=\abs{w_1}^{2(m_1+\cdots+m_l)}\lambda(z,w)$.

On the other hand, arguing as in the proof of \cite[Theorem 1.1]{C}, we see that the weighted Bergman projection on $\mathbb{H}_{\phi_j}^N$ with the weight $\lambda$ is $L^p\big(\mathbb{H}_{\phi_j}^N,\lambda\big)$ bounded if and only if $\mathcal{B}_{\mathbb{H}^N,\lambda}$ is $L^p\big(\mathbb{H}^N,\lambda\big)$. Applying Lemma \ref{LpH*} to $\mathcal{B}_{\mathbb{H}^{n*}, \tilde{\lambda}}$, we obtain the desired result.
\end{proof}

\section{The Two-weight Inequality}
\label{sec4}

In this section, we mainly focus on the two-weight inequality \eqref{2weight}, and prove Proposition \ref{mu1>mu2}.

\subsection{Basic Definitions}

Throughout this section, $z$ and $w$ will denote complex variables in $\mathbb{R}_+^2$. The letter $c$ will denote some positive constant independent of the variables and the functions in the context. In some equations, different constants will be absorbed into a single $c$.

\begin{definition}
In this section, $\mathcal{B}$ will denote the \emph{Bergman projection on $\mathbb{R}_+^2$}. So we have\footnote{We use the normalization $\mbox{Area}(\mathbb{D})=1$ as in the previous section.}
\[
\mathcal{B}(f)(z)=\int_{\mathbb{R}_+^2}-\frac{1}{(z-\overline{w})^2}f(w)\,dA(w)
\]
for all $f\in C_{c}^{\infty}(\mathbb{R}_+^2)$. We also consider the \emph{``absolute value" operator} $\widetilde{\mathcal{B}}$ of $\mathcal{B}$, which is defined as
\[
\widetilde{\mathcal{B}}(f)(z)=\int_{\mathbb{R}_+^2}\frac{1}{\abs{z-\overline{w}}^2}f(w)\,dA(w),
\]
where we replace the kernel $-\frac{1}{(z-\overline{w})^2}$ by its absolute value.
\end{definition}

\begin{notation}
For a weight $\mu$, a measurable function $f$, and any measurable set $W$ with its Lebesgue measure $\abs{W}$, we may use the notation
\[
\mu(W)=\int_W\mu(z)\,dA(z)
\]
and
\[
\dashint_Wf(z)\,dA(z)=\frac{1}{\abs{W}}\int_Wf(z)\,dA(z).
\]
\end{notation}

To prove Proposition \ref{mu1>mu2}, we give the definition of two variants of the $A_p$ class, which was introduced by Muckenhoup.\footnote{See \cite{Mu}, \cite{LS1}, and also \cite[Chapter 5]{S}.}

\begin{definition}
\label{Apclass}
For $p>1$, let $p'$ denote the conjugate exponent of $p$. We say the two weights $\mu_1$ and $\mu_2$ are in the $A_p(\mathbb{R}_+^2)$ class denoted by $(\mu_1,\mu_2)\in A_p(\mathbb{R}_+^2)$ if there is a positive constant $c$ so that
\[
\dashint_{D\cap\mathbb{R}_+^2}\mu_1(z)\,dA(z)\left(\dashint_{D\cap\mathbb{R}_+^2}\mu_2(z)^{-\frac{p'}{p}}\,dA(z)\right)^{\frac{p}{p'}}\le c,
\]
for all disks $D$ centered at $z\in\overline{\mathbb{R}_+^2}$. A disk is said to be a \emph{special disk} if it is centered at some $x\in\mathbb{R}$. We say $(\mu_1,\mu_2)\in A_p^+(\mathbb{R}_+^2)$ if the above inequality holds for all special disks. For some weight $\mu$, if $(\mu,\mu)\in A_p^+(\mathbb{R}_+^2)$ (resp. $A_p(\mathbb{R}_+^2)$), we may simply adopt the notation $\mu\in A_p^+(\mathbb{R}_+^2)$ (resp. $A_p(\mathbb{R}_+^2)$). 
\end{definition}

\begin{remark}
\label{Aprmk}
The class $A_p^+(\mathbb{R}_+^2)$ is strictly wider than the class $A_p(\mathbb{R}_+^2)$. For the one-weight case see the comments following the definition of $A_p^+$ in \cite{LS1}. For two-weight case see Proposition \ref{Ap+widerAp} in section \ref{sec5}.
\end{remark}

As in \cite{LS1}, we now introduce a standard ``tiling" of $\mathbb{R}_+^2$ and the associated averaging operator (the ``conditional expectation").

\begin{definition}
The \emph{standard ``tiling"} of $\mathbb{R}_+^2$ are the squares $\{S_{j,k}\}$ of form
\[
S_{j,k}=\{z=x+iy\in\mathbb{C}:\ 2^k\le y\le 2^{k+1}\,\,\,\mbox{and}\,\,\,j\cdot2^k\le x\le (j+1)\cdot2^{k+1}\}
\]
for all $j,k\in\mathbb{Z}$. Note that each $S_{j,k}$ has side-length $2^k$, the interiors of $S_{j,k}$'s are disjoint, and $\mathbb{R}_+^2=\bigcup_{j,k\in\mathbb{Z}}S_{j,k}$.

Define the associated \emph{averaging operator} $E$ by
\[
E(f)(z)=\dashint_{S_{j,k}}f(z)\,dA(z),\,\,\,\,\mbox{if}\,\,\,\,z\in S_{j,k},
\]
for any nonnegative measurable function $f$ on $\mathbb{R}_+^2$.
\end{definition}

\subsection{Some Properties of the Operator $E$}

Here we introduce some important properties of the operator $E$.

\begin{lemma}
\label{Eproperty}
We have the following basic properties of $E$. For any nonnegative measurable functions $f$ and $g$, letting $p'$ be the conjugate exponent of $p$, we have
\begin{enumerate}
\item[(a)]
{$\int_{\mathbb{R}_+^2}E(f)(z)g(z)\,dA(z)=\int_{\mathbb{R}_+^2}E(f)(z)E(g)(z)\,dA(z)$,}
\item[(b)]
{$\int_{\mathbb{R}_+^2}\left(E(f)(z)\right)^pg(z)\,dA(z)\le\int_{\mathbb{R}_+^2}E(f^p)(z)g(z)\,dA(z)$,}
\item[(c)]
{$E(fg)(z)\le\left(E(f^{p})(z)\right)^{1/p}(E(g^{p'})(z))^{1/p'}\,\,\,\,\mbox{for all}\,\,\,z\in\mathbb{R}_+^2$.}
\end{enumerate}
\end{lemma}

\begin{proof}
The verification is straightforward.
\end{proof}

\begin{lemma}
For any $f\ge0$, we have
\begin{equation}
\label{B<EBE}
\widetilde{\mathcal{B}}(f)\le cE\widetilde{\mathcal{B}}E(f).
\end{equation}
\end{lemma}

\begin{proof}
See \cite[Eq. (4.9)]{LS1} and also \cite[Proposition 4.2.2]{Chen}.
\end{proof}

\begin{proposition}
\label{Ap+toAp}
If $(\mu_1,\mu_2)\in A_p^+(\mathbb{R}_{+}^2)$, then $(E(\mu_1),E(\mu_2))\in A_p(\mathbb{R}_{+}^2)$.
\end{proposition}

\begin{proof}
The proof of the two-weight case is essentially the same as the proof of the one-weight case in \cite[Proposition 4.6]{LS1}. See also \cite[Proposition 4.3.1]{Chen}.
\end{proof}

\subsection{The Two-weight Inequality}

Now we are ready to prove Proposition \ref{mu1>mu2}.

\begin{proof}[Proof of Proposition \ref{mu1>mu2}.]
For the sufficient part, since $(\mu_1,\mu_2)\in A_p^+(\mathbb{R}_{+}^2)$ and $c\mu_1\ge\mu_2$, it is easy to see that $\mu_1\in A_p^+(\mathbb{R}_+^2)$. By Proposition \ref{Ap+toAp}, we have $E(\mu_1)\in A_p(\mathbb{R}_+^2)$.

For $f\ge0$, as in the proof of \cite[Proposition 4.5]{LS1}, we have
\begin{equation}
\label{eq4.9}
\begin{split}
\int_{\mathbb{R}_+^2}\left(\widetilde{\mathcal{B}}(f)(z)\right)^p\mu_1(z)\,dA(z)
&\le c\int_{\mathbb{R}_+^2}\left(E\widetilde{\mathcal{B}}E(f)(z)\right)^p\mu_1(z)\,dA(z)\\
&\le c\int_{\mathbb{R}_+^2}E\left(\widetilde{\mathcal{B}}E(f)(z)\right)^p\mu_1(z)\,dA(z)\\
&=c\int_{\mathbb{R}_+^2}\left(\widetilde{\mathcal{B}}E(f)(z)\right)^pE(\mu_1)(z)\,dA(z)\\
&\le c\int_{\mathbb{R}_+^2}\left(E(f)(z)\right)^pE(\mu_1)(z)\,dA(z),
\end{split}
\end{equation}
where the first line follows from \eqref{B<EBE}, the second line follows from Lemma \ref{Eproperty} (b), the third line follows from Lemma \ref{Eproperty} (a), and the last line follows from $E(\mu_1)\in A_p(\mathbb{R}_+^2)$ and the fact that $\widetilde{\mathcal{B}}$ can be represented as a singular integral\footnote{See \cite[Section 3]{LS1} and also \cite[Proposition 4.4.1]{Chen}.}.

By Lemma \ref{Eproperty} (c), we see
\begin{equation}
\label{eq4.10}
E(f)(z)=E(f\mu_2^{\frac{1}{p}}\cdot\mu_2^{-\frac{1}{p}})(z)\le\bigg(E(f^p\mu_2)(z)\bigg)^{\frac{1}{p}}\left(E(\mu_2^{-\frac{p'}{p}})(z)\right)^{\frac{1}{p'}}.
\end{equation}

For any standard ``tiling" square $S_{j,k}$, let $x_0$ be the real part of the center of $S_{j,k}$ and let $D=D_{2^{k+2}}(x_0)$ be the special disk of radius $2^{k+2}$ centered at $x_0$. It is easy to see that $S_{j,k}\subset D\cap\mathbb{R}_+^2$. Since $\frac{1}{\abs{S_{j,k}}}=\frac{8\pi}{\abs{D\cap\mathbb{R}_+^2}}$ and $(\mu_1,\mu_2)\in A_p^+(\mathbb{R}_+^2)$, for $z\in S_{j,k}$ we see that
\begin{equation}
\label{eq4.11}
\begin{split}
&\;\;\;\;\;E(\mu_1)(z)\left(E(\mu_2^{-\frac{p'}{p}})(z)\right)^{\frac{p}{p'}}\\
&\le c\dashint_{D\cap\mathbb{R}_+^2}\mu_1(z)\,dA(z)\left(\dashint_{D\cap\mathbb{R}_+^2}\mu_2(z)^{-\frac{p'}{p}}\,dA(z)\right)^{\frac{p}{p'}}\\
&\le c
\end{split}
\end{equation}
for some $c>0$ independent of $S_{j,k}$.

Therefore combining \eqref{eq4.9}, \eqref{eq4.10} and \eqref{eq4.11}, we see that
\begin{align*}
\int_{\mathbb{R}_+^2}\left(\widetilde{\mathcal{B}}(f)(z)\right)^p\mu_1(z)\,dA(z)
&\le c\int_{\mathbb{R}_+^2}E(f^p\mu_2)(z)\left(E(\mu_2^{-\frac{p'}{p}})(z)\right)^{\frac{p}{p'}}E(\mu_1)(z)\,dA(z)\\
&\le c\sum_{j,k}\int_{S_{j,k}}E(f^p\mu_2)(z)\,dA(z)\\
&=c\int_{\mathbb{R}_+^2}f(z)^p\mu_2(z)\,dA(z).
\end{align*}

For the necessary part, following the same idea as in the proof of \cite[Proposition 3.4]{LS1} and using the fact that $c\mu_1\ge\mu_2$ for some $c>0$, we arrive at\footnote{Since $\mu_2^{-p'/p}$ is not assumed to be continuous here, the limit argument following \cite[Eq. (3.11)]{LS1} will not apply to our situation.}
\begin{equation}
\label{eq4.15}
\mu_1(D\cap\mathbb{R}_+^2)\left(\dashint_{D\cap\mathbb{R}_+^2}f(z)\,dA(z)\right)^p\le c\int_{D\cap\mathbb{R}_+^2}f(z)^p\mu_2(z)\,dA(z),
\end{equation}
for some $c>0$, provided $f\chi_{D\cap\mathbb{R}_+^2}\in L^p(\mathbb{R}_+^2,\mu_2)$.

To show that \eqref{eq4.15} is indeed the $A_p^+(\mathbb{R}_+^2)$ condition, we argue as in the proof of the necessary part of \cite[Theorem 1]{Mu}. Suppose that $\int_{D\cap\mathbb{R}_+^2}\mu_2(z)^{-\frac{p'}{p}}\,dA(z)=\infty$. Then, by duality of the space $L^p(D\cap\mathbb{R}_+^2)$, there is a $g\in L^p(D\cap\mathbb{R}_+^2)$, so that $\int_{D\cap\mathbb{R}_+^2}g(z)\mu_2(z)^{-\frac{1}{p}}\,dA(z)=\infty$. Take $f=g\mu_2^{-\frac{1}{p}}\chi_{D\cap\mathbb{R}_+^2}$ in \eqref{eq4.15}. Then $\int_{D\cap\mathbb{R}_+^2}f(z)\,dA(z)=\infty$ and $\int_{D\cap\mathbb{R}_+^2}f(z)^p\mu_2(z)\,dA(z)<\infty$. So \eqref{eq4.15} gives $\mu_1(D\cap\mathbb{R}_+^2)=0$, which contradicts to the assumption $\mu_1>0$ almost everywhere. So we see that $\int_{D\cap\mathbb{R}_+^2}\mu_2(z)^{-\frac{p'}{p}}\,dA(z)<\infty$.

Now we can take $f=(\mu_2)^{-\frac{p'}{p}}\chi_{D\cap\mathbb{R}_+^2}$ in \eqref{eq4.15}. We have
\[
\int_{D\cap\mathbb{R}_+^2}f(z)^p\mu_2(z)\,dA(z)=\int_{D\cap\mathbb{R}_+^2}\mu_2(z)^{-\frac{p'}{p}}\,dA(z)<\infty,
\]
since $\frac{p'}{p}=p'-1$. Hence we see that \eqref{eq4.15} implies the $A_p^+(\mathbb{R}_+^2)$ condition.
\end{proof}

\section{A Wider Class of Weights}
\label{sec5}

In this section, we will apply Proposition \ref{mu1>mu2} to prove Theorem \ref{LpHwider}.

\subsection{An Observation}

We first have the following observation which is mentioned in Remark \ref{Aprmk}.

\begin{proposition}
\label{Ap+widerAp}
For $z\in\mathbb{R}_+^2$, $k\in\mathbb{Z}$, $s\in(0,2]$ and $p>1$, suppose
\[
\mu_1(z)=\abs{\frac{i-z}{i+z}}^{-(k+1)p+s+2k}
\]
and
\[
\mu_2(z)=\abs{\frac{i-z}{i+z}}^{(1-s-k)p+s+2k},
\]
then $(\mu_1,\mu_2)\notin A_p(\mathbb{R}_+^2)$ for $s\neq2$. But we have $(\mu_1,\mu_2)\in A_p^+(\mathbb{R}_+^2)$ if and only if $s+2k+2>(k+1)p$ and $p(s+k+1)>s+2k+2$.
\end{proposition}

\begin{proof}
To show $(\mu_1,\mu_2)\notin A_p(\mathbb{R}_+^2)$, we consider any disk $D_{\epsilon}(i)$ centered at $i$ with radius $\epsilon<\frac{1}{2}$. For $z\in D_{\epsilon}(i)$, since $\abs{i-z}<\epsilon<\frac{1}{2}$, we see that $\frac{3}{2}\le\abs{i+z}\le\frac{5}{2}$. So, by Definition \ref{Apclass}, we only need to look at 
\[
\dashint_{D_{\epsilon}(i)}\abs{i-z}^{-(k+1)p+s+2k}\,dA(z)\left(\dashint_{D_{\epsilon}(i)}\abs{i-z}^{-\frac{p'}{p}[(1-s-k)p+s+2k]}\,dA(z)\right)^{p/p'}.
\]
Assuming both integrands are integrable, we obtain $\epsilon^{(s-2)p}$. But $s\in(0,2)$ and $p>1$, so we see the quantity above tends to $\infty$ as $\epsilon\to0$.

To show $(\mu_1,\mu_2)\in A_p^+(\mathbb{R}_{+}^2)$, we consider two integrals
\[
I_1=\frac{1}{\abs{D\cap\mathbb{R}_+^2}}\int_{D\cap\mathbb{R}_+^2}\abs{\frac{i-z}{i+z}}^{-(k+1)p+s+2k}\,dA(z),
\]
and
\[
I_2=\frac{1}{\abs{D\cap\mathbb{R}_+^2}}\int_{D\cap\mathbb{R}_+^2}\abs{\frac{i-z}{i+z}}^{-\frac{p'}{p}[(1-s-k)p+s+2k]}\,dA(z),
\]
where $D=D_R(x_0)$ is any special disk with radius $R$ centered at $x_0\in\mathbb{R}$. Let $D_0=D_{\frac{1}{2}}(i)$ be the disk with radius $\frac{1}{2}$ centered at $i$. We separate our arguments into two cases.

Case (I), $R<\frac{1}{2}$.

It is easy to see that $D\cap D_0=\varnothing$ hence $\abs{i-z}>\frac{1}{2}$. Note that, as $\abs{z}\to\infty$, $\abs{\frac{i-z}{i+z}}\to1$, so there is an $M$ such that, when $\abs{z}>M$, $1\ge\abs{\frac{i-z}{i+z}}\ge\frac{1}{2}$. But when $\abs{z}\le M$, $\abs{i+z}\le M+1$, so $1\ge\abs{\frac{i-z}{i+z}}\ge\frac{1}{2(M+1)}$. Therefore, the integrands in $I_1$ and $I_2$ are bounded above by some constants that are independent of the special disk $D$. Then $I_1I_2^{\frac{p}{p'}}\le c$, for some $c>0$.

Case (II), $R\ge\frac{1}{2}$.

We split both $I_1$ and $I_2$ into two integrals respectively, one integrates over $D\cap\mathbb{R}_+^2\setminus D_0$ and the other integrates over $D\cap\mathbb{R}_+^2\cap D_0$. For the same reasoning as in case (I), the parts integrated over $D\cap\mathbb{R}_+^2\setminus D_0$ is bounded. The parts integrated over $D\cap\mathbb{R}_+^2\cap D_0$ are bounded respectively by
\[
\frac{8}{\pi}\int_{D_0}\abs{\frac{i-z}{i+z}}^{-(k+1)p+s+2k}\,dA(z),
\]
and
\[
\frac{8}{\pi}\int_{D_0}\abs{\frac{i-z}{i+z}}^{-\frac{p'}{p}[(1-s-k)p+s+2k]}\,dA(z).
\]
Since $\abs{i-z}\le\frac{1}{2}$ for $z\in D_0$, we see that $\frac{3}{2}\le\abs{i+z}\le\frac{5}{2}$, so the two integrals above are bounded respectively by
\[
c\int_{D_0}\abs{i-z}^{-(k+1)p+s+2k}\,dA(z)
\]
and
\[
c\int_{D_0}\abs{i-z}^{-\frac{p'}{p}[(1-s-k)p+s+2k]}\,dA(z).
\]
The first integral above is bounded by a constant if and only if $-(k+1)p+s+2k+2>0$, and the second is bounded if and only if $-\frac{p'}{p}[(1-s-k)p+s+2k]+2>0$. Solving the two inequalities, we see that $s+2k+2>(k+1)p$ and $p(s+k+1)>s+2k+2$.
\end{proof}

\begin{remark}
Combining Proposition \ref{Ap+widerAp} and Proposition \ref{mu1>mu2}, we see that the two-weight $A_p$ condition is not a necessary condition\footnote{Compare to the general Calder\'{o}n-Zygmund type singular integral, see \cite[Chapter 5]{S}.} for the inequality \eqref{2weight}. The reason is that the Bergman kernel is not ``singular" in $\mathbb{R}_+^2$, but only on the boundary.
\end{remark}

\subsection{Alternative Proof of Proposition \ref{LpD*}.}

As an application of Proposition \ref{mu1>mu2}, we give an alternative proof of Proposition \ref{LpD*}.

\begin{proof}[Alternative proof of Proposition \ref{LpD*}.]
For $s'=s+2k$, from \eqref{kernel} we know that
\[
B_{s'}(z,\zeta)=\frac{s}{2}(z\overline{\zeta})^{-(k+1)}B_0(z,\zeta)+\left(1-\frac{s}{2}\right)(z\overline{\zeta})^{-k}B_0(z,\zeta)
\]
for $z,\zeta\in\mathbb{D}^*$. So we may write $\mathcal{B}_{s'}=\frac{s}{2}\mathcal{T}_1+\big(1-\frac{s}{2}\big)\mathcal{T}_2$, where $\mathcal{T}_1$ is the operator associated to the kernel $(z\overline{\zeta})^{-(k+1)}B_0(z,\zeta)$ with the weight $\abs{\zeta}^{s'}$ and $\mathcal{T}_2$ is the operator associated to the kernel $(z\overline{\zeta})^{-k}B_0(z,\zeta)$ with the weight $\abs{\zeta}^{s'}$.

For the operator $\mathcal{T}_1$, showing its boundedness is the same as showing
\[
\int_{\mathbb{D}^*}\abs{\int_{\mathbb{D}^*}\frac{(z\overline{\zeta})^{-(k+1)}f(\zeta)}{(1-z\overline{\zeta})^2}\abs{\zeta}^{s'}\,dA(\zeta)}^p\abs{z}^{s'}\,dV(z)\le C\int_{\mathbb{D}^*}\abs{f(z)}^p\abs{z}^{s'}\,dA(z).
\]
By the Cayley transform $\varphi:\mathbb{R}_+^2\to\mathbb{D}$, where $\varphi(z)=\frac{i-z}{i+z}$, we see that the above inequality is equivalent to
\[
\int_{\mathbb{R}_+^2-\{i\}}\abs{\int_{\mathbb{R}_+^2-\{i\}}-\frac{\tilde{f}(w)\,dV(w)}{(z-\overline{w})^2}}^p\mu_1(z)\,dA(z)\le C\int_{\mathbb{R}_+^2-\{i\}}\abs{\tilde{f}(z)}^p\mu_2(z)\,dA(z),
\]
where $\tilde{f}(z)=f(\varphi(z))\varphi(z)^{k+1}\abs{\varphi(z)}^{s-2}\cdot\frac{1}{(i+1)^2}$, $\mu_1(z)=4\abs{\frac{i-z}{i+z}}^{-(k+1)p+s+2k}\cdot\abs{\frac{1}{(i+z)^2}}^{2-p}$ and $\mu_2(z)=4\abs{\frac{i-z}{i+z}}^{(1-s-k)p+s+2k}\cdot\abs{\frac{1}{(i+z)^2}}^{2-p}$. This is exactly \eqref{2weight}, except that $\mu_1$ and $\mu_2$ are only locally integrable on $\mathbb{R}_+^2\setminus\{i\}$---not on $\mathbb{R}_+^2$. However, since $\abs{\frac{i-z}{i+z}}\le 1$ and $s\le2$, we have $\mu_1(z)/\mu_2(z)=\abs{(i-z)/(i+z)}^{(s-2)p}\ge1$, so the local integrability of $\mu_2$ will be guaranteed by the local integrability of $\mu_1$ at $i$ as we will see below. Moreover, by the relation $\mu_1\ge\mu_2$, Proposition \ref{mu1>mu2} applies.\footnote{From the proof of the necessity, we see the boundedness of $\mathcal{T}_1$ will imply the local integrability of $\mu_1$ and $\mu_2$ at $i$.} So we see that $\mathcal{T}_1$ is bounded if and only if $(\mu_1,\mu_2)\in A_p^+(\mathbb{R}_{+}^2)$.

For the condition $(\mu_1,\mu_2)\in A_p^+(\mathbb{R}_{+}^2)$, we first note that $\sigma(z)=\abs{\frac{1}{(i+z)^2}}^{2-p}\in A_p^+(\mathbb{R}_+^2)$ for all $p>1$. To see this, from the classical result $\mathcal{B}_0$ is $L^p$ bounded for all $p>1$, where $\mathcal{B}_0$ is the ordinary Bergman projection on $\mathbb{D}^*$ which is the same as the ordinary Bergman projection on the unit disk. Then \eqref{2weight} holds with both $\mu_1$ and $\mu_2$ replaced by $\sigma$, and hence $\sigma\in A_p^+(\mathbb{R}_+^2)$ for all $p>1$ by Proposition \ref{mu1>mu2}.

As in the proof of Proposition \ref{Ap+widerAp}, we consider two integrals
\[
I_1=\frac{1}{\abs{D\cap\mathbb{R}_+^2}}\int_{D\cap\mathbb{R}_+^2}4\abs{\frac{i-z}{i+z}}^{-(k+1)p+s+2k}\cdot\abs{\frac{1}{(i+z)^2}}^{2-p}\,dA(z)
\]
and
\[
I_2=\frac{1}{\abs{D\cap\mathbb{R}_+^2}}\int_{D\cap\mathbb{R}_+^2}\left(4\abs{\frac{i-z}{i+z}}^{(1-s-k)p+s+2k}\cdot\abs{\frac{1}{(i+z)^2}}^{2-p}\right)^{-\frac{p'}{p}}\,dA(z),
\]
where $D=D_R(x_0)$ is again any special disk with radius $R$ centered at $x_0\in\mathbb{R}$.

For $R<1/2$, the same argument as in the proof of Proposition \ref{Ap+widerAp} case (I), shows that the integrands in $I_1$ and $I_2$ are bounded above by $c\sigma$ and $c\sigma^{-\frac{p'}{p}}$ respectively. So this case is done by the fact that $\sigma\in A_p^+(\mathbb{R}_+^2)$ for all $p>1$.

For $R\ge1/2$, the same argument as in the proof of Proposition \ref{Ap+widerAp} case (II) shows that we only need to consider whether 
\[
\abs{i-z}^{-(k+1)p+s+2k}\,\,\,\,\,\mbox{and}\,\,\,\,\,\abs{i-z}^{-\frac{p'}{p}[(1-s-k)p+s+2k]}
\]
are locally integrable at $i$. This shows the local integrabilities of $\mu_1$ and $\mu_2$, and we see that this is true if and only if $s+2k+2>(k+1)p$ and $p(s+k+1)>s+2k+2$ by Proposition \ref{Ap+widerAp}.

Denoting by $U_1=\{p\in(1,\infty)\,:\,s+2k+2>(k+1)p\,\,\,\mbox{and}\,\,\,p(s+k+1)>s+2k+2\}$ the range for $p$, it is not difficult to see that $U_1$ is an open interval. So we obtain $(\mu_1,\mu_2)\in A_p^+(\mathbb{R}_{+}^2)$ and hence $\mathcal{T}_1$ is bounded if and only if $p\in U_1$.

Similarly, $\mathcal{T}_2$ is bounded if and only if \eqref{2weight} holds for the two weights $\mu_1(z)=4\sigma(z)\abs{\frac{i-z}{i+z}}^{-kp+s+2k}$ and $\mu_2(z)=4\sigma(z)\abs{\frac{i-z}{i+z}}^{-(s+k)p+s+2k}$. Again, $\mu_1$ and $\mu_2$ may not be locally integrable. Fortunately, we have $\mu_1(z)/\mu_2(z)=\abs{(i-z)/(i+z)}^{sp}\le1$, since $\abs{(i-z)(i+z)}\le 1$ and $s>0$. So we do not need both of $\mu_1$ and $\mu_2$ to be locally integrable, indeed, they will not be in some cases. We instead apply Proposition \ref{mu1>mu2} to a single weight, either to $\mu_1$ or to $\mu_2$. Then by $\mu_1\le\mu_2$ we get the desired inequality. So $\mathcal{T}_2$ is bounded if $\mu_1\in A_p^+(\mathbb{R}_{+}^2)\cap L_{\text{loc}}^1(\mathbb{R}_+^2)$ or $\mu_2\in A_p^+(\mathbb{R}_{+}^2)\cap L^1_{\text{loc}}(\mathbb{R}_+^2)$.

Following a similar argument, $\mu_j\in A_p^+(\mathbb{R}_{+}^2)$ will guarantee the local integrability of $\mu_j$, for $j=1,2$. Then we see, by listing all possibilities of $k\in\mathbb{Z}$, that $\mu_1\in A_p^+(\mathbb{R}_{+}^2)$ or $\mu_2\in A_p^+(\mathbb{R}_{+}^2)$ if and only if $p\in U_2=\{p\in(1,\infty)\,|\,s+2k+2>pk\,\,\,\mbox{and}\,\,\,(s+k+2)p>s+2k+2\}$ for $s\neq2$. But for $s=2$ we do not need to worry about $\mathcal{T}_2$, since $\mathcal{B}_{s'}=\frac{s}{2}\mathcal{T}_1+\big(1-\frac{s}{2}\big)\mathcal{T}_2$. It is not hard to see that $U_2$ is also an open interval, and we have $\mathcal{T}_2$ is bounded if $p\in U_2$.

Now, if both $\mathcal{T}_1$ and $\mathcal{T}_2$ are bounded, then $\mathcal{B}_{s'}$ is bounded. Since a simple argument shows that $U_1\subset U_2$ properly, we see that $\mathcal{B}_{s'}$ is bounded if $p\in U_1$. Conversely, if we look at some endpoint $p$ of $U_1$, then $p\notin U_1$ but $p\in U_2$. In this case, we see that $\mathcal{T}_1$ is unbounded, $\mathcal{T}_2$ is bounded, and hence $\mathcal{B}_{s'}$ is unbounded. So by interpolation we see that $\mathcal{B}_{s'}$ is unbounded for all $p\notin U_1$.

Therefore, for $p>1$, $\mathcal{B}_{s'}$ is bounded if and only if $p\in U_1$. When $s'\in(0,\infty)$, $U_1=\big(\frac{s+2k+2}{s+k+1},\frac{s+2k+2}{k+1}\big)$. When $s'\in[-3,0]$, $U_1=(1,\infty)$. When $s'\in(-4,-3)$, $U_1=\big(2-s,\frac{2-s}{1-s}\big)$. When $s'=-4$, $U_1=(1,\infty)$. When $s'\in(-\infty,-4)$, $U_1=\big(\frac{s+2k+2}{k+1},\frac{s+2k+2}{s+k+1}\big)$.
\end{proof}

\begin{remark}
Besides Proposition \ref{mu1>mu2}, the analysis for $\mathcal{T}_2$ to be bounded also supports our Conjecture \ref{conjecture1}, since the ``effective" range for $p$ is obtained by checking $(\mu_1,\mu_2)\in A_p^+(\mathbb{R}_+^2)$.
\end{remark}

\subsection{A Wider Class of Weights.}

Before considering the Hartogs triangle, we first look at the punctured disk $\mathbb{D}^*$ with a wider class of weights.

\begin{proposition}
\label{LpD*wider}
Assume that $p>1$. Let $\mu(z)=\abs{z}^{s'}\abs{g(z)}^2$, where $z\in\mathbb{D}^*$, $s'\in\mathbb{R}$, and $g$ is a non-vanishing holomorphic function on $\mathbb{D}$. Suppose that the weighted Bergman projection $\mathcal{B}_{\mathbb{D},\abs{g}^2}$ on $\mathbb{D}$ with the weight $\abs{g}^2$ is $L^p\big(\mathbb{D},|g|^2\big)$ bounded\footnote{The range of $p$ for $\mathcal{B}_{\mathbb{D},\abs{g}^2}$ to be $L^p$-bounded must be an open interval. See \cite{LS1,Z} and \cite[Chapter 5]{S} for a full consideration.} if and only if $p\in(p_0,p_0')$ for some $p_0\ge1$, and suppose that the weighted Bergman projection $\mathcal{B}_{s'}$ on $\mathbb{D}^*$ with the weight $\lambda(z)=\abs{z}^{s'}$, where $z\in\mathbb{D}^*$, is $L^p(\mathbb{D}^*,\lambda)$ bounded if and only if $p\in(p_1,p_1')$ for some $p_1\ge1$ as in Proposition \ref{LpD*}. Then the weighted Bergman projection $\mathcal{B}_{\mathbb{D}^*,\mu}$ on $\mathbb{D}^*$ with the weight $\mu$ is $L^p(\mathbb{D}^*,\mu)$ bounded if $p\in(p_0,p_0')\cap(p_1,p_1')$.

In addition, if $(p_1,p_1')\subset(p_0,p_0')$ properly, then $\mathcal{B}_{\mathbb{D}^*,\mu}$ is $L^p(\mathbb{D}^*,\mu)$ bounded if and only if $p\in(p_1,p_1')$.
\end{proposition}

\begin{proof}
It is not difficult to see that the weighted Bergman kernel $B_{\mathbb{D},\abs{g}^2}(z,\zeta)$ associated to $\mathbb{D}$ with the weight $\abs{g}^2$ can be expressed as\footnote{See, for example, \cite[Theorem 3.4]{Z}.}
\[
B_{\abs{g}^2}(z,\zeta)=\frac{1}{g(z)\overline{g(\zeta)}}\frac{1}{(1-z\overline{\zeta})^2},
\]
where $(z,\zeta)\in\mathbb{D}\times\mathbb{D}$. Similar argument as in the alternative proof of Proposition \ref{LpD*} shows that $\mathcal{B}_{\mathbb{D},\abs{g}^2}$ is $L^p$-bounded if and only if \eqref{2weight} holds with both $\mu_1$ and $\mu_2$ replaced by $\sigma(z)=\abs{g(\varphi(z))\cdot\frac{1}{(i+z)^2}}^{2-p}$. By Proposition \ref{mu1>mu2}, we see that $\sigma\in A_p^+(\mathbb{R}_+^2)$ if and only if $\mathcal{B}_{\mathbb{D},\abs{g}^2}$ is $L^p$-bounded, that is, if and only if $p\in(p_0,p_0')$.

Now we turn to the weight $\mu(z)=\abs{z}^{s'}\abs{g(z)}^2$ on $\mathbb{D}^*$. The same argument above applies to the weighted Bergman projection $\mathcal{B}_{\mathbb{D}^*,\mu}$, whose associated weighted Bergman kernel is
\[
B_{\mathbb{D}^*,\mu}(z,\zeta)=\frac{1}{g(z)\overline{g(\zeta)}}B_{s'}(z,\zeta),
\]
where $B_{s'}(z,\zeta)$ is the weighted Bergman kernel associated to $\mathbb{D}^*$ with the weight $\lambda$. Then by the relation $\mathcal{B}_{s'}=\frac{s}{2}\mathcal{T}_1+\left(1-\frac{s}{2}\right)\mathcal{T}_2$ in the alternative proof of Proposition \ref{LpD*}, $\mathcal{B}_{\mathbb{D}^*,\mu}$ is $L^p(\mathbb{D}^*,\mu)$ bounded if \eqref{2weight} holds both for the first pair
\[
\mu_1(z)=4\sigma(z)\abs{\frac{i-z}{i+z}}^{-(k+1)p+s+2k},\,\,\,\,\mu_2(z)=4\sigma(z)\abs{\frac{i-z}{i+z}}^{(1-s-k)p+s+2k}
\]
and for the second pair
\[
\mu_1(z)=4\sigma(z)\abs{\frac{i-z}{i+z}}^{-kp+s+2k},\,\,\,\,\mu_2(z)=4\sigma(z)\abs{\frac{i-z}{i+z}}^{-(s+k)p+s+2k}.
\]
Following the same argument as in the alternative proof of Proposition \ref{LpD*}, and noting that $g$ is bounded above and below on the disk $D_{\frac{1}{2}}(i)$ with radius $\frac{1}{2}$ centered at $i$, we see that \eqref{2weight} holds for the first pair if and only if $p\in(p_0,p_0')\cap(p_1,p_1')$. Similarly, \eqref{2weight} holds for the second pair if $p\in(p_0,p_0')\cap U$, for some larger open interval $U$ which contains $p_1$ and $p_1'$.\footnote{If $(p_1,p_1')=(1,\infty)$, then the conclusion is trivial.} Therefore $\mathcal{B}_{\mathbb{D}^*,\mu}$ is $L^p(\mathbb{D}^*,\mu)$ bounded if $p\in(p_0,p_0')\cap(p_1,p_1')$.

If, in addition, $(p_1,p_1')\subset(p_0,p_0')$ properly then, for $p=p_1$ and $p=p_1'$, \eqref{2weight} fails for the first pair, but holds for the second pair. So $\mathcal{B}_{\mathbb{D}^*,\mu}$ is unbounded for these $p$s. Hence $\mathcal{B}_{\mathbb{D}^*,\mu}$ is bounded if and only if $p\in(p_1,p_1')$.
\end{proof}

\subsection{The Hartogs Triangle.}

Now we can consider the Hartogs triangle $\mathbb{H}$ with a wider class of weights of the form in \eqref{weight} and prove Theorem \ref{LpHwider}.

\begin{proof}[Proof of Theorem \ref{LpHwider}.]
This is a direct consequence of Proposition \ref{inflation_theorem} and Proposition \ref{LpD*wider}.
\end{proof}

\section{$L^p$ Regularity with Two Weights}
\label{sec6}

In this section, we consider the $L^p$ regularity of the weighted Bergman projection on $\mathbb{H}$ mapping from one weighted space $L^p\big(\mathbb{H},\abs{z_2}^{s'}\big)$ to the other $L^p\big(\mathbb{H},\abs{z_2}^{t}\big)$.

\subsection{The $L^p$ Boundedness.}

By applying Propositon \ref{mu1>mu2}, we can prove Theorem \ref{LpHst}.

\begin{proof}[Proof of Theorem \ref{LpHst}.]
The boundedness of the mapping is equivalent to
\[
\int_{\mathbb{H}}\abs{\mathcal{B}_{\mathbb{H},s'}(f)(z)}^p\abs{z_2}^t\,dV(z)\le C\int_{\mathbb{H}}\abs{f(z)}^p\abs{z_2}^{s'}\,dV(z).
\]
By Corollary \ref{transform_kernel}, and considering the biholomorphism $\Phi:\mathbb{H}\to\mathbb{D}\times\mathbb{D}^*$ via $\Phi(z)=\big(\frac{z_1}{z_2},z_2\big)$, we see that the above inequality is equivalent to
\begin{align*}
&\,\,\,\,\,\,\,\int_{\mathbb{D}\times\mathbb{D}^*}\abs{\int_{\mathbb{D}\times\mathbb{D}^*}B_0\otimes B_{s'}(z,\zeta)\tilde{f}(\zeta)\abs{\zeta_2}^{s'}\,dV(\zeta)}^p\abs{z_2}^{t+2-p}\,dV(z)\\
&\le C\int_{\mathbb{D}\times\mathbb{D}^*}\abs{\tilde{f}(z)}^p\abs{z_2}^{s'+2-p}\,dV(z),
\end{align*}
where $\tilde{f}(z)=f(\Phi^{-1}(z))z_2$, $z\in\mathbb{D}\times\mathbb{D}^*$, $B_{s'}$ is the weighted Bergman kernel associated to $\mathbb{D}^*$ with the weight $\abs{w}^{s'}$, $w\in\mathbb{D}^*$, and $B_0$ is the Bergman kernel associated to $\mathbb{D}$. By Lemma \ref{product_operator}, and the fact that the ordinary Bergman projection on the unit disk is $L^p$ bounded for all $p>1$, we see that the above inequality is equivalent to
\begin{equation}
\label{ineqD*}
\begin{split}
&\,\,\,\,\,\,\,\int_{\mathbb{D}^*}\abs{\int_{\mathbb{D}^*}B_{s'}(z,\zeta)\tilde{f}(\zeta)\abs{\zeta}^{s'}\,dA(\zeta)}^p\abs{z}^{t+2-p}\,dA(z)\\
&\le C\int_{\mathbb{D}^*}\abs{\tilde{f}(z)}^p\abs{z}^{s'+2-p}\,dA(z).
\end{split}
\end{equation}
By the relation $\mathcal{B}_{s'}=\frac{s}{2}\mathcal{T}_1+\left(1-\frac{s}{2}\right)\mathcal{T}_2$ in the alternative proof of Proposition \ref{LpD*}, we see that $\mathcal{T}_1$ is bounded if \eqref{2weight} holds for the first pair:
\[
\mu_1(z)=4\sigma(z)\abs{\frac{i-z}{i+z}}^{-(k+2)p+t+2},\,\,\,\,\mu_2(z)=4\sigma(z)\abs{\frac{i-z}{i+z}}^{-(s+k)p+s+2k+2};
\]
and $\mathcal{T}_2$ is bounded if \eqref{2weight} holds for the second pair (with $\mu_1$ and $\mu_2$ replaced by $\tilde{\mu}_1$ and $\tilde{\mu}_2$ respectively in \eqref{2weight}):
\[
\tilde{\mu}_1(z)=4\sigma(z)\abs{\frac{i-z}{i+z}}^{-(k+1)p+t+2},\,\,\,\,\tilde{\mu}_2(z)=4\sigma(z)\abs{\frac{i-z}{i+z}}^{-(s+k+1)p+s+2k+2},
\]
where $\sigma(z)=\abs{\frac{1}{(i+z)^2}}^{2-p}\in A_p^+(\mathbb{R}_+^2)$ for all $p>1$ as we have already seen in the alternative proof of Proposition \ref{LpD*}.

Case (I), $t\ge s'$. Then $\big(\frac{s+2k+4}{s+k+2},\frac{t+4}{k+2}\big)$ is nonempty.

For $\mathcal{T}_1$, when $\mu_1\ge\mu_2$, that is, $t-s'\le(2-s)p$, Proposition \ref{mu1>mu2} applies to the first pair $\mu_1$ and $\mu_2$. It tells us that $\mathcal{T}_1$ is bounded if and only if $p\in\big(\frac{s+2k+4}{s+k+2},\frac{t+4}{k+2}\big)$.

For $\mathcal{T}_1$, when $\mu_1<\mu_2$, that is, $t-s'>(2-s)p$, Proposition \ref{mu1>mu2} applies either to $\mu_1$ or to $\mu_2$. It tells us that $\mathcal{T}_1$ is bounded if $p\in\big(\frac{s+2k+4}{s+k+2},\frac{s+2k+4}{s+k}\big)\bigcup\big(\frac{t+4}{k+4},\frac{t+4}{k+2}\big)$.

For $\mathcal{T}_2$, when $\tilde{\mu}_1\ge\tilde{\mu}_2$, that is, $t-s'+sp\le0$, Proposition \ref{mu1>mu2} applies to the second pair $\tilde{\mu}_1$ and $\tilde{\mu_2}$. It tells us that $\mathcal{T}_2$ is bounded if and only if $p\in\big(\frac{s+2k+4}{s+k+3},\frac{t+4}{k+1}\big)$.

For $\mathcal{T}_2$, when $\tilde{\mu}_1<\tilde{\mu}_2$, that is, $t-s'+sp>0$, Proposition \ref{mu1>mu2} applies either to $\tilde{\mu}_1$ or to $\tilde{\mu}_2$. It tells us that $\mathcal{T}_2$ is bounded if $p\in\big(\frac{s+2k+4}{s+k+3},\frac{s+2k+4}{s+k+1}\big)\bigcup\big(\frac{t+4}{k+3},\frac{t+4}{k+1}\big)$.

Since $\frac{s+2k+4}{s+k+2}\in\big(\frac{s+2k+4}{s+k+3},\frac{s+2k+4}{s+k+1}\big)$ and $\frac{t+4}{k+2}\in\big(\frac{t+4}{k+3},\frac{t+4}{k+1}\big)$, if $p=\frac{s+2k+4}{s+k+2}+\epsilon$ or $p=\frac{t+4}{k+2}-\epsilon$ for any sufficiently small $\epsilon>0$, then $\mathcal{B}_{\mathbb{H},s'}$ is $L^p$-bounded. By interpolation theorem, we see that $\mathcal{B}_{\mathbb{H},s'}$ is bounded if $p\in\big(\frac{s+2k+4}{s+k+2},\frac{t+4}{k+2}\big)$.

In addition, if $t-s'\le(2-s)p$, then by looking at $p=\frac{s+2k+4}{s+k+2}$ or $p=\frac{t+4}{k+2}$, we see that $\mathcal{B}_{\mathbb{H},s'}$ is unbounded. Hence $\mathcal{B}_{\mathbb{H},s'}$ is bounded if and only if $p\in\big(\frac{s+2k+4}{s+k+2},\frac{t+4}{k+2}\big)$.

Case (II), $t<s'$. Then $t-s'\le(2-s)p$ and $L^p\big(\mathbb{H},\abs{z_2}^{t}\big)\subset L^p\big(\mathbb{H},\abs{z_2}^{s'}\big)$.

If $p\notin\big(\frac{s+2k+4}{s+k+2},\frac{s+2k+4}{k+2}\big)$, then by Theorem \ref{LpH} (1), $\mathcal{B}_{\mathbb{H},s'}$ is unbounded on $L^p\big(\mathbb{H},\abs{z_2}^{s'}\big)$. Hence $\mathcal{B}_{\mathbb{H},s'}$ is unbounded from $L^p\big(\mathbb{H},\abs{z_2}^{s'}\big)$ to $L^p\big(\mathbb{H},\abs{z_2}^{t}\big)$.

Given any $p\in\big(\frac{s+2k+4}{s+k+2},\frac{s+2k+4}{k+2}\big)$, we look at the boundary condition $\frac{t_0+4}{k+2}=\frac{s+2k+4}{s+k+2}$. For this $t_0$, we have $A^p\big(\mathbb{H},\abs{z_2}^{t_0}\big)\subsetneq A^p\big(\mathbb{H},\abs{z_2}^{s'}\big)$. Otherwise, we would have $A^p\big(\mathbb{H},\abs{z_2}^{t'}\big)=A^p\big(\mathbb{H},\abs{z_2}^{s'}\big)$ for some $t'>t_0$ such that $\frac{s+2k+4}{s+k+2}<\frac{t'+4}{k+2}<p$. Then by the same argument in Case (I) applying to $t'$, we see that $\mathcal{B}_{\mathbb{H},s'}$ is unbounded from $L^p\big(\mathbb{H},\abs{z_2}^{s'}\big)$ to $A^p\big(\mathbb{H},\abs{z_2}^{t'}\big)$. However, by Theorem \ref{LpH} (1), $\mathcal{B}_{\mathbb{H},s'}$ is bounded on $L^p\big(\mathbb{H},\abs{z_2}^{s'}\big)$ and hence maps $L^p\big(\mathbb{H},\abs{z_2}^{s'}\big)$ into $A^p\big(\mathbb{H},\abs{z_2}^{t'}\big)$. This is not possible.

So for all $t\le t_0$, that is, $\big(\frac{s+2k+4}{s+k+2},\frac{t+4}{k+2}\big)=\varnothing$, we have $A^p\big(\mathbb{H},\abs{z_2}^{t}\big)\subsetneq A^p\big(\mathbb{H},\abs{z_2}^{s'}\big)$. Since $\mathcal{B}_{\mathbb{H},s'}$ is identity on the analytic subspace $A^p\big(\mathbb{H},\abs{z_2}^{s'}\big)$, we see that $\mathcal{B}_{\mathbb{H},s'}$ is unbounded from $L^p\big(\mathbb{H},\abs{z_2}^{s'}\big)$ to $L^p\big(\mathbb{H},\abs{z_2}^{t}\big)$.

If $t>t_0$, then we see that $\big(\frac{s+2k+4}{s+k+2},\frac{t+4}{k+2}\big)$ is nonempty. By the same argument in Case (I), we see that $\mathcal{B}_{\mathbb{H},s'}$ is bounded if and only if $p\in\big(\frac{s+2k+4}{s+k+2},\frac{t+4}{k+2}\big)$.

Combining Case(I) and Case(II), we have proved the first part of the Theorem. On the other hand, if $p\le\frac{s+2k+4}{s+k+2}$. We look at the boundary condition $p=\frac{s+2k'+2}{s+k'+1}$, where $k'=k+1\ge0$. Now the same argument in the original proof of the unboundedness part of Proposition \ref{LpD*} applies to \eqref{ineqD*}. It is easy to see, for the same sequence $\{f_n\}$ with $k$ replaced by $k'$, that $f_n$ is $L^p\big(\mathbb{D}^*,\abs{z}^{s'}\big)$-bounded uniformly on $n$ and the $L^p\big(\mathbb{D}^*,\abs{z}^{t}\big)$-norm of $\mathcal{B}_{s'}(f_n)$ blows up as $n\to\infty$ for any choice of $t\in\mathbb{R}$. This completes the proof.
\end{proof}

\subsection{A Sharp Estimate.}

By applying Theorem \ref{LpHst}, it is not difficult to obtain the following sharp $L^p$ estimate of the weighted Bergman projection on the Hartogs triangle $\mathbb{H}$ with the weight $\abs{z_2}^{s'}$.

\begin{corollary}
\label{sharpestimate}
Let $\mathcal{B}_{\mathbb{H},s'}$ be the weighted Bergman projection on the Hartogs triangle $\mathbb{H}$ with the weight $\abs{z_2}^{s'}$, where $z\in\mathbb{H}$ and $s'\in\mathbb{R}$ with the unique expression $s'=s+2k$ for $k\in\mathbb{Z}$ and $s\in(0,2]$. Assume that $p>1$ and $k\ge-1$.
\begin{enumerate}
\item{Given any $p\ge\frac{s+2k+4}{k+2}$, let $t=t(\epsilon)=p(k+2)-4+\epsilon$ for any $\epsilon>0$. Then $\mathcal{B}_{\mathbb{H},s'}$ is bounded from $L^p\big(\mathbb{H},\abs{z_2}^{s'}\big)$ to $L^p\big(\mathbb{H},\abs{z_2}^{t}\big)$.}

\item{If $p\in(\frac{s+2k+4}{s+k+2},\frac{s+2k+4}{k+2})$, then $\mathcal{B}_{\mathbb{H},s'}$ is bounded from $L^p\big(\mathbb{H},\abs{z_2}^{s'}\big)$ onto $A^p\big(\mathbb{H},\abs{z_2}^{s'}\big)$.}

\item{If $p\le\frac{s+2k+4}{s+k+2}$, then $\mathcal{B}_{\mathbb{H},s'}$ is unbounded from $L^p\big(\mathbb{H},\abs{z_2}^{s'}\big)$ to $L^p\big(\mathbb{H},\abs{z_2}^{t}\big)$ for any choice of $t\in\mathbb{R}$.}

\end{enumerate}
\end{corollary}

\begin{proof}
It suffices to prove (1). By Theorem \ref{LpHst}, we see that $\mathcal{B}_{\mathbb{H},s'}$ is bounded from $L^p\big(\mathbb{H},\abs{z_2}^{s'}\big)$ to $L^p\big(\mathbb{H},\abs{z_2}^{t}\big)$ if $p<\frac{t+4}{k+2}$. This is trivially true if $t=p(k+2)-4+\epsilon$ for any $\epsilon>0$.
\end{proof}

\begin{remark}
In Corollary \ref{sharpestimate} (1), if we look at the boundary condition $p=\frac{s+2k+4}{k+2}$, then $t(\epsilon)=s'+\epsilon$. So this tells us, for $p=\frac{s+2k+4}{k+2}$, that $\mathcal{B}_{\mathbb{H},s'}$ can always map $L^p\big(\mathbb{H},\abs{z_2}^{s'}\big)$ into a slightly bigger space $L^p\big(\mathbb{H},\abs{z_2}^{s'+\epsilon}\big)$, but $\mathcal{B}_{\mathbb{H},s'}$ does not map $L^p\big(\mathbb{H},\abs{z_2}^{s'}\big)$ into itself. However, we can say nothing about the surjectivity of $\mathcal{B}_{\mathbb{H},s'}$ in this case.
\end{remark}

\section{Concluding Remarks}

1. We have proved the $L^p$ regularity of the weighted Bergman projection on the Hartogs triangle for a certain interval of $p$. This interval is sharp. From this, it is natural to ask: what space does the weighted Bergman projection map to when $p$ is out of this certain interval? In particular, when $p$ equals the right endpoint of this interval, Corollary \ref{sharpestimate} tells us that it should be some space between $L^p\big(\mathbb{H},\abs{z_2}^{s'}\big)$ and $L^p\big(\mathbb{H},\abs{z_2}^{s'+\epsilon}\big)$ for $\epsilon>0$ sufficiently small.

2. It is well-known that the Bergman projection on the Hartogs triangle does not preserve the Sobolev spaces. But the result for the weighted Bergman projection here provides an idea to study the $L^p$-Sobolev mapping property of the Bergman projection---put some weight on the target space. We will turn to this matter in an upcoming paper.

3. It seems likely that the model $\mathbb{D}^*$ with the weight $\lambda(z)=\abs{z}^{s'}$ provides an idea to determine the type of the product-like ``$\abs{z_1}<\abs{z_2}<\cdots$" boundary singularity. Proposition \ref{LpD*} plays a significant role in studying the $L^p$ mapping property of the weighted Bergman projection. We hope that we could make this idea more clear and describe the type of such singularity in the future.

4. Proposition \ref{mu1>mu2} provides a partial result of Conjecture \ref{conjecture1}. We suspect that this type of ``singular integral" (sometimes called ``Hilbert integral") is closely related to some ``special maximal operator"  which is centered along the boundary. We will go further in this direction in another paper. We hope that we can prove Conjecture \ref{conjecture1} and generalize it to higher dimension in the future.

\bibliographystyle{plain}

\end{document}